\documentclass[11pt]{amsart}
\usepackage{fullpage,url,amssymb}
\usepackage{mathrsfs}
\numberwithin{equation}{section}
\usepackage{color}

\newcommand{\defi}[1]{\textsf{#1}} 

\newenvironment{romanenum}{\hfill \begin{enumerate} }{\end{enumerate}}
\newenvironment{alphenum}{\hfill \begin{enumerate} }{\end{enumerate}}

\newcommand{\FF}{{\mathbb F}}
\newcommand{\GG}{{\mathbb G}}
\newcommand{\ZZ}{{\mathbb Z}}

\def\bbar#1{\setbox0=\hbox{$#1$}\dimen0=.2\ht0 \kern\dimen0 \overline{\kern-\dimen0 #1}}

\newcommand{\Lbar}{{\bbar{L}} }

\newcommand{\kbar}{{\bbar{k}}}

\newcommand{\Fbar}{\bbar{F}}

\newcommand{\p}{{\mathfrak p}}
\newcommand{\aA}{{\mathfrak a}}

\newcommand{\calB}{{\mathcal B}}
\newcommand{\calC}{{\mathcal C}}

\newcommand{\calS}{{\mathcal S}}

\newcommand{\calV}{{\mathcal V}}

\newcommand{\OO}{{\mathcal O}}

\DeclareMathOperator{\tr}{tr}

\DeclareMathOperator{\Tr}{Tr}

\DeclareMathOperator{\Frob}{Frob}

\DeclareMathOperator{\End}{End}

\DeclareMathOperator{\Hom}{Hom}

\DeclareMathOperator{\Aut}{Aut}
\DeclareMathOperator{\Gal}{Gal}

\DeclareMathOperator{\ord}{ord}

\newcommand{\laurent}{(\!(t^{-1})\!) }

\newcommand{\perf}{{\operatorname{perf}}}

\newcommand{\ab}{{\operatorname{ab}}}
\newcommand{\sep}{{\operatorname{sep}}}

\newcommand{\lin}{{\operatorname{lin}}}

\newcommand{\GL}{\operatorname{GL}}
\newcommand{\SL}{\operatorname{SL}}
\newcommand{\PGL}{\operatorname{PGL}}
\newcommand{\PSL}{\operatorname{PSL}}

\newcommand{\del}{\partial}

\newcommand{\Cent}{\operatorname{Cent}}

\DeclareMathOperator{\Ad}{Ad}

\newcommand{\ang}[1]{\langle #1\rangle}

\def\QQ{\mathbb Q}
\def\RR{\mathbb R}
\def\CC{\mathbb C}

\def\bfA{\mathbf A}

\def\twistedseries{[\![\tau^{-1}]\!] }

\def\twistedLaurent{ (\!(\tau^{-1})\!) }

\newtheorem{theorem}{Theorem}[section]
\newtheorem{lemma}[theorem]{Lemma}
\newtheorem{corollary}[theorem]{Corollary}
\newtheorem{proposition}[theorem]{Proposition}
\theoremstyle{definition}
\newtheorem{definition}[theorem]{Definition}

\newtheorem{example}[theorem]{Example}

\theoremstyle{remark}
\newtheorem{remark}[theorem]{Remark}

\definecolor{webcolor}{rgb}{0,0,1}
\definecolor{webbrown}{rgb}{.6,0,0}
\usepackage[
       	colorlinks,
       	linkcolor=webbrown,  filecolor=webcolor,  citecolor=webbrown, 
       	backref,
       	backref,
       	pdfauthor={David Zywina},
	pdftitle={The Sato-Tate Law for Drinfeld Modules},
]{hyperref}
\usepackage[alphabetic,backrefs,lite]{amsrefs} 

\begin{document}
\title[The Sato-Tate Law for Drinfeld Modules]{The Sato-Tate Law for Drinfeld Modules}
\author{David Zywina}
\address{Department of Mathematics and Statistics, Queen's University, Kingston, ON  K7L~3N6, Canada}
\email{zywina@mast.queensu.ca}
\urladdr{http://www.mast.queensu.ca/\~{}zywina}
\date{\today}

\begin{abstract} 

We prove an analogue of the Sato-Tate conjecture for Drinfeld modules.   Using ideas of Drinfeld, J.-K.~Yu showed that Drinfeld modules satisfy some Sato-Tate law, but did not describe the actual law.   More precisely, for a Drinfeld module $\phi$ defined over a field $L$, he constructs a continuous representation $\rho_\infty\colon W_L \to D^\times$ of the Weil group of $L$ into a certain division algebra, which encodes the Sato-Tate law.    When $\phi$ has generic characteristic and $L$ is finitely generated, we shall describe the image of $\rho_\infty$ up to commensurability.    As an application, we give improved upper bounds for the Drinfeld module analogue of the Lang-Trotter conjecture.
\end{abstract}

\maketitle
\section{Introduction}

\subsection{Notation} \label{SS:notation}
We first set some notation that will hold throughout.  Let $F$ be a global function field.   Let $k$ be its field of constants and denote by $q$ the cardinality of $k$.  Fix a place $\infty$ of $F$ and let $A$ be the subring consisting of those functions that are regular away from $\infty$.    For each place $\lambda$ of $F$, let $F_\lambda$ be the completion of $F$ at $\lambda$.   Let $\ord_\lambda$ denote the corresponding discrete valuation on $F_\lambda$, $\OO_\lambda$ the valuation ring, and $\FF_\lambda$ the residue field.   Let $d_\infty$ be the degree of the extension $\FF_\infty/k$.\\

For a field extension $L$ of $k$, let $\Lbar$ be a fixed algebraic closure and let $L^\sep$ be the separable closure of $L$ in $\Lbar$.  We will denote the algebraic closure of $k$ in $\Lbar$ by $\kbar$.    
Let $\Gal_L=\Gal(L^\sep/L)$ be the absolute Galois group of $L$.   The \defi{Weil group} $W_L$ is the subgroup of $\Gal_L$ consisting of those $\sigma$ for which $\sigma|_{\kbar}$ is an integral power  $\deg(\sigma)$ of the Frobenius automorphism $x\mapsto x^q$.    The map $\deg\colon W_L \to \ZZ$ is a group homomorphism.  Denote by $L^\perf$ the perfect closure of $L$ in $\Lbar$.\\

Let $L[\tau]$ be the twisted polynomial ring with the commutation rule $\tau \cdot a = a^q \tau$ for $a\in L$; in particular, $L[\tau]$ is non-commutative if $L\neq k$. Identifying $\tau$ with $X^q$, we find that $L[\tau]$ is the ring of $k$-linear additive polynomials in $L[X]$ where multiplication corresponds to composition of polynomials.   Suppose further that $L$ is perfect.   Let $L\twistedLaurent$ be the skew-field consisting of twisted Laurent series in $\tau^{-1}$   (we need $L$ to be perfect so that $\tau^{-1}\cdot  a = a^{1/q} \tau$ holds).    Define the valuation $\ord_{\tau^{-1}}\colon L\twistedLaurent\to \ZZ\cup\{+\infty\}$ by $\ord_{\tau^{-1}}(\sum_i a_i \tau^{-i}) = \inf\{ i : a_i \neq 0 \}$ and $\ord_{\tau^{-1}}(0)=+\infty$.   The valuation ring of $\ord_{\tau^{-1}}$ is $L\twistedseries$, i.e., the ring of twisted formal power series in $\tau^{-1}$.

For a ring $R$ and a subset $S$, let $\Cent_R(S)$ be the subring of $R$ consisting of those elements that commute with $S$.

\subsection{Drinfeld module background and the Sato-Tate law}  \label{SS:Sato-Tate law}
A \defi{Drinfeld module} over a field $L$ is a ring homomorphism 
\[
\phi\colon A \to L[\tau], \, a\mapsto \phi_a
\] 
such that $\phi(A)$ is not contained in the subring of constant polynomials.   Let $\del \colon L[\tau] \to L$ be the ring homomorphism $\sum_{i}b_i \tau^i \mapsto b_0$.  The \defi{characteristic} of $\phi$ is the kernel of $\del\circ \phi \colon A \to L$; it is a prime ideal of $A$.  If the characteristic of $\phi$ is the zero ideal,  then we say that $\phi$ has \defi{generic characteristic}.    Using $\del\circ\phi$, we shall view $L$ as an extension of $k$, and as an extension of $F$ when $\phi$ has generic characteristic.  

The ring $L[\tau]$ is contained in the skew field $L^\perf\twistedLaurent$.  The map $\phi$ is injective, so it extends uniquely to a homomorphism $\phi\colon F\hookrightarrow L^\perf\twistedLaurent$.   The function $v\colon F\to \ZZ\cup\{+\infty\}$ defined by $v(x)=\ord_{\tau^{-1}}(\phi_x)$ is a non-trivial discrete valuation that satisfies $v(x)\leq 0$ for all non-zero $x\in A$.  Therefore $v$ is equivalent to $\ord_\infty$, and hence there exists a positive $n\in\QQ$ that satisfies
\begin{equation} \label{E:degrees}
\ord_{\tau^{-1}}(\phi_x) = n d_\infty \ord_\infty(x) 
\end{equation}
for all $x\in F^\times$.   The number $n$ is called the \defi{rank} of $\phi$ and it is always an integer.  Since $L^\perf\twistedLaurent$ is complete with respect to $\ord_{\tau^{-1}}$, the map $\phi$ extends uniquely to a homomorphism
\[
\phi\colon F_\infty\hookrightarrow L^\perf\twistedLaurent
\]
that satisfies (\ref{E:degrees}) for all $x\in F_\infty^\times$.    This is the starting point for the constructions of Drinfeld in \cite{MR0439758}.  Let $\FF_\infty\to L^\perf$ be the homomorphism obtained by composing $\phi|_{\FF_\infty}$ with the map that takes an element in $L^\perf\twistedseries$ to its constant term.   So $\phi$ induces an embedding of $\FF_\infty$ into $L^\perf$, and hence into $L$ itself.\\

Let $D_\phi$ be the centralizer of $\phi(A)$, equivalently of $\phi(F_\infty)$, in $\Lbar\twistedLaurent$.  The ring $D_\phi$ is an $F_\infty$-algebra via our extended $\phi$.  We shall see in \S\ref{S:construction} that $D_\phi$ is a central $F_\infty$-division algebra with invariant $-1/n$.  For each field extension $L'/L$, the ring $\End_{L'}(\phi)$ of \defi{endomorphisms} of $\phi$ is the centralizer of $\phi(A)$ in $L'[\tau]$.   We have inclusions $\phi(A)\subseteq \End_{\Lbar}(\phi) \subseteq D_\phi$.

Following J.-K.~Yu \cite{MR2018826}, we shall define a continuous homomorphism
\[
\rho_\infty \colon W_L \to D_\phi^\times
\]
that, as we will explain, should be thought of as the Sato-Tate law for $\phi$.   Let us briefly describe the construction, see \S\ref{S:construction} for details.  There exists an element $u\in \Lbar\twistedLaurent^\times$ with coefficients in $L^\sep$ such that $u^{-1}\phi(A)u \subseteq \kbar\twistedLaurent$.  
For $\sigma\in W_L$, we define
\[
\rho_\infty(\sigma):=\sigma(u) \tau^{\deg(\sigma)} u^{-1}
\]
where $\sigma$ acts on the series $u$ by acting on its coefficients.  We will verify in \S\ref{S:construction} that $\rho_{\infty}(\sigma)$ belongs to $D_\phi^\times$, is independent of the initial choice of $u$, and that $\rho_\infty$ is indeed a continuous homomorphism.
Our construction of $\rho_\infty$ varies slightly from than that of Yu's (cf.~\S\ref{SS:difference in construction}); his representation $\rho_\infty$ is only canonically defined up to an inner automorphism.   When needed, we will make the dependence on the Drinfeld module clear by using the notation $\rho_{\phi,\infty}$ instead of  $\rho_{\infty}$.\\

Now consider a Drinfeld module $\phi\colon A \to L[\tau]$ of rank $n$ with generic characteristic and assume that $L$ is a finitely generated field.   Choose an integral scheme $X$ of finite type over $k$ with function field $L$.  For a closed point $x$ of $X$, denote its residue field by $\FF_x$.   Using that $A$ is finitely generated, we may replace $X$ with an open subscheme such that the coefficients of all elements of $\phi(A)\subseteq L[\tau]$ are integral at each closed point $x$ of $X$.  By reducing the coefficients of $\phi$, we obtain a homomorphism
\[
\phi_x \colon A \to \FF_x[\tau].
\]
After replacing $X$ by an open subscheme, we may assume further that $\phi_x$ is a Drinfeld module of rank $n$ for each closed point $x$ of $X$.  

Let $P_{\phi,x}(T)\in A[T]$ be the \defi{characteristic polynomial} of the Frobenius endomorphism $\pi_x:=\tau^{[\FF_x:k]} \in \End_{\FF_x}(\phi_x)$; it is the degree $n$ polynomial that is a power of the minimal polynomial of $\pi_x$ over $F$.    We shall see that $\rho_\infty$ is unramified at $x$ and that
\[
P_{\phi,x}(T)=\det(TI - \rho_\infty(\Frob_x) )
\]
where we denote by $\det\colon D_\phi \to F_\infty$ the reduced norm.   The representation $\rho_{\infty}$ can thus be used to study the distribution of the coefficients of the polynomials $P_{\phi,x}(T)$ with respect to the $\infty$-adic topology.    Though Yu showed that $\phi$ satisfies an analogue of Sato-Tate, he was unable to say what the Sato-Tate law actually was.   We shall address this by describing the image of $\rho_\infty$ up to commensurability.   We first consider the case where $\phi$ has no extra endomorphisms.

\begin{theorem}\label{T:main theorem non-CM}
Let $\phi\colon A\to L[\tau]$ be a Drinfeld module with generic characteristic where $L$ is a finitely generated field and assume that $\End_{\Lbar}(\phi)=\phi(A)$.   The group $\rho_\infty(W_L)$ is an open subgroup of finite index in $D_\phi^\times$.
\end{theorem}

We will explain the corresponding equidistribution result in \S\ref{SS:equidistribution} after a brief interlude on elliptic curves in \S\ref{SS:elliptic curves}.    

Now consider a general Drinfeld module $\phi\colon A\to L[\tau]$ with generic characteristic, $L$ finitely generated, and no restriction on the endomorphism ring of $\phi$.   The reader may safely read ahead under the assumption that $\End_{\Lbar}(\phi)=\phi(A)$ (indeed, a key step in the proof is to reduce to the case where $\phi$ has no extra endomorphisms).

The ring $\End_\Lbar(\phi)$ is commutative and a projective module over $A$ with rank $m\leq n$, cf.~\cite{MR0384707}*{p.569~Corollary}.  Also, $E_\infty:= \End_{\Lbar}(\phi)\otimes_A F_\infty$ is a field of degree $m$ over $F_\infty$.   Let $B_\phi$ be the centralizer of $\End_{\Lbar}(\phi)$, equivalently of $E_\infty$, in $\Lbar\twistedLaurent$; it is central $E_\infty$-division algebra with invariant $-m/n$.

There is a finite separable extension $L'$ of $L$ for which $\End_{\Lbar}(\phi)=\End_{L'}(\phi)$.  We shall see that $\rho_\infty(W_{L'})$ commutes with $\End_{L'}(\phi)$, and hence $\rho_{\infty}(W_{L'})$ is a subgroup of $B_\phi^\times$.   The following generalization of Theorem~\ref{T:main theorem non-CM} says that after this constraint it taken into account, the image of $\rho_\infty$ is, up to finite index, as large as possible.

\begin{theorem}\label{T:main theorem}
Let $\phi\colon A\to L[\tau]$ be a Drinfeld module with generic characteristic where $L$ is a finitely generated field.   The group $\rho_\infty(W_L)\cap B_\phi^\times$ is an open subgroup of finite index in $B_\phi^\times$.  Moreover, the groups $\rho_\infty(W_L)$ and $B_\phi^\times$ are commensurable.
\end{theorem}

These theorems address several of the questions raised by J.K.~Yu in \cite{MR2018826}*{\S4}.

\subsection{Elliptic curves} \label{SS:elliptic curves}

We now recall the Sato-Tate conjecture for elliptic curves over a number field.   We shall present it in a manner so that the analogy with Drinfeld modules is transparent; in particular, this strengthens the analogy presented in \cite{MR2018826}.  

Let $\mathbf{H}$ be the real quaternions; it is a central $\RR$-division algebra with invariant $-1/2$.   We will denote the reduced norm by $\det\colon \mathbf{H}\to \RR$.  Let $\mathbf{H}_1$ be the group of quaternions of norm 1.   

For a group $H$, we shall denote the set of conjugacy classes by $H^\sharp$.  Now suppose that $H$ is a compact topological group and let $\mu$ be the Haar measure on $H$ normalized so that $\mu(H)=1$.  Using the natural map $f\colon H\to H^\sharp$, we give $H^\sharp$ the quotient topology.  The \defi{Sato-Tate measure} on $H^\sharp$ is the measure $\mu_{\textrm{ST}}$ for which $\mu_{\textrm{ST}}(U)=\mu(f^{-1}(U))$ for all open subsets $U\supseteq H^\sharp$.\\

Fix an elliptic curve $E$ defined over a number field $L$, and let $S$ be the set of non-zero prime ideals of $\OO_L$ for which $E$ has bad reduction.   For each non-zero prime ideal $\p\notin S$ of $\OO_L$, let $E_\p$ be the elliptic curve over $\FF_\p=\OO_L/\p$ obtained by reducing $E$ modulo $\p$, and let $\pi_\p$ be the Frobenius endomorphism of $E_\p/\FF_\p$.    The \defi{characteristic polynomial} of $\pi_\p$ is the polynomial $P_{E,\p}(T)\in \QQ[T]$ of degree 2 that is a power of the minimal polynomial of $\pi_\p$ over $\QQ$.   We have $P_{E,\p}(T)=T^2-a_\p(E)T + N(\p)$ where $N(\p)$ is the cardinality of $\FF_\p$ and $a_\p(E)$ is an integer that satisfies $|a_\p(E)|\leq 2N(\p)^{1/2}$.\\

Suppose that $E/L$ does not have complex multiplication, that is, $\End_{\Lbar}(E)=\ZZ$.       For each prime $\p\notin S$, there is a unique conjugacy class $\theta_\p$ of $\mathbf{H}^\times$ such that $P_{E,\p}(T)=\det(TI-\theta_\p)$ (this uses that $a_\p(E)^2-4N(\p)\leq 0$).   We can normalize these conjugacy classes by defining $\vartheta_\p=\theta_\p/\sqrt{N(\p)}$; it is the unique conjugacy class of $\mathbf{H}_1$ for which $\det(TI-\vartheta_\p)=T^2-(a_\p(E)/\sqrt{N(\p)})T+1$.    The \defi{Sato-Tate conjecture} for $E/L$ predicts that the conjugacy classes $\{\vartheta_\p\}_{\p\notin S}$ are equidistributed in $\mathbf{H}_1^\sharp$ with respect to the Sato-Tate measure, i.e., for any continuous function $f\colon \mathbf{H}_1^\sharp \to \CC$, we have
 \[
 \lim_{x\to +\infty}  \frac{1}{|\{\p \notin S: N(\p)\leq x\}|} \sum_{\p\notin S,\, N(\p)\leq x}  f(\vartheta_\p) = \int_{\mathbf{H}_1^\sharp} f(\xi) d\mu_{\textrm{ST}}(\xi).
 \]
Note that $\mathbf{H}_1$ and $\operatorname{SU}_2(\CC)$ are maximal compact subgroups of $(\mathbf{H}\otimes_\RR\CC)^\times \cong \GL_2(\CC)$ and are thus conjugate.   So our quaternion formulation agrees with the more familiar version dealing with conjugacy classes in $\operatorname{SU}_2(\CC)$.   [The Sato-Tate conjecture has been proved in the case where $L$ is totally real and $E$ has non-integral $j$-invariant, cf.~\cite{MR2470687,MR2470688,MR2630056}]

\begin{remark}
The analogous case is a Drinfeld module $\phi\colon A \to L[\tau]$ with generic characteristic and rank 2  where $L$ is a global function field.   The algebra $D_\phi$ is a central $F_\infty$-division algebra with invariant $-1/2$.  For each place $\p\neq \infty$ of good reduction, there is a unique conjugacy class $\theta_\p$ of $D_\phi^\times$ such that $\det(TI-\theta_\p)=P_{\phi,\p}(T)$.    This information is all encoded in our function $\rho_\infty$, since $\theta_\p$ is the conjugacy class containing $\rho_\infty(\Frob_\p)$.   We will discuss the equidistribution law in \S\ref{SS:equidistribution}; it will be a consequence of the function field version of the Chebotarev density theorem.
\end{remark}

Now suppose that $E/L$ has complex multiplication, and assume that $R:=\End_{\Lbar}(E)$ equals $\End_L(E)$.  The ring $R$ is an order in the quadratic imaginary field $K:=R\otimes_\ZZ \QQ$.  For $\p\notin S$, reduction of endomorphism rings modulo $\p$ induces an injective homomorphism $K \hookrightarrow \End_{\FF_\p}(E_{\p})\otimes_\ZZ\QQ$ whose image contains $\pi_\p$; let $\theta_\p$ be the unique element of $K$ that maps to $\pi_\p$.    

From the theory of complex multiplication, there is a continuous homomorphism 
\[
\rho_{E,\infty} \colon W_L \to (K\otimes_\QQ {\RR})^\times = (\End_L(E)\otimes_\ZZ \RR)^\times
\]
such that $\rho_{E,\infty}(\Frob_\p)=\theta_\p$ for all $\p\notin S$, where $W_L$ is the Weil group of $L$; see \cite{MR563921}*{Chap.~1 \S8}.  (Using the Weil group here is rather excessive; the image is abelian, so the representation factors through $W_L^\ab$ which in turn is isomorphic to the idele class group of $L$.)  Choose an isomorphism $\CC=K\otimes_\QQ {\RR}$.  We can normalize by defining $\vartheta_\p=\theta_\p/\sqrt{N(\p)}$ which belongs to the group $\mathbf{S}$ of complex numbers with absolute value $1$.   Then the \defi{Sato-Tate law} for $E/L$ says that the elements $\{\vartheta_\p\}_{\p\notin S}$ are equidistributed in $\mathbf{S}$.  This closely resembles the case where $\phi$ is a Drinfeld module of rank 2 and $\End_{L}(\phi)$ has rank 2 over $A$; we then have a continuous homomorphism $\rho_\infty\colon W_L \to B_\phi^\times=(\End_{L}(\phi)\otimes_A F_\infty)^\times$.

\subsection{Equidistribution law} \label{SS:equidistribution}
Let $\phi\colon A \to L[\tau]$ be a Drinfeld module of rank $n$.   To ease notation, set $D=D_\phi$.    Let $\OO_D$ be the valuation ring of $D$ with respect to the valuation $\ord_{\tau^{-1}}\colon D \to \ZZ\cup \{+\infty\}$.   The continuous homomorphism $\rho_\infty \colon W_L \to D^\times$ induces a continuous representation
\[
\widehat{\rho}_\infty\colon \Gal_L \to \widehat{D^\times}
\]
where $\widehat{D^\times}$ is the profinite completion of $D^\times$.

Now suppose that $L$ is finitely generated and that $\End_{\Lbar}(\phi)=\phi(A)$ (similar remarks will hold without the assumption on the endomorphism ring).        Choose a scheme $X$ as in \S\ref{SS:Sato-Tate law} and let $|X|$ be its set of closed points.  For a subset $\calS$ of $|X|$, define $F_\calS(s)= \sum_{x\in \calS} N(x)^{-s}$ where $N(x)$ is the cardinality of the residue field $\FF_x$.   The \defi{Dirichlet density} of $\calS$ is the value $\lim_{s\to d^+} F_\calS(s)/F_{|X|}(s)$, assuming the limit exists, where $d$ is the transcendence degree of $L$ (see \cite{MR1474696}*{Appendix~B} for details on Dirichlet density). \\

Let $\mu$ be the Haar measure on $H:=\widehat{\rho}_\infty(\Gal_L)$ normalized so that $\mu(H)=1$.    Take an open subset $U$ of $H$ that is stable under conjugation.   The {Chebotarev density theorem} then implies that the set 
\[
\{ x\in|X| \,: \, \widehat{\rho}_\infty(\Frob_x) \subseteq U\}
\]
has Dirichlet density $\mu(U)$, cf.~\cite{MR2018826}*{Corollary~3.5}.    This equidistribution law can be viewed as the analogue of Sato-Tate.    The choice of $X$ is not important since different choices will agree away from a set of points with density 0.

Theorem~\ref{T:main theorem non-CM} implies that the group $H$ is an open subgroup of $\widehat{D^\times}$.  So for a ``random'' $x\in |X|$, the element $\rho_\infty(\Frob_x)$ will resemble a random conjugacy class of $H$, and hence a rather generic element of $\widehat{D^\times}$.    

Fix a closed subgroup $V$ of $F_\infty^\times$ that does not lie in $\OO_\infty^\times$.  That $V$ is unbounded in the $\infty$-adic topology implies that the quotient group $D^\times/V$ is compact.  So as a quotient of $\widehat{\rho}_\infty$, we obtain a Galois representation $\tilde\rho\colon\Gal_L\to D^\times/V.$   The image $\tilde\rho(\Gal_L)$ is thus an open subgroup of finite index in $D^\times/V$ and as above, the Chebotarev density theorem gives an equidistribution law in terms of Dirichlet density.   These representations can be viewed as analogues of the normalization process described in \S\ref{SS:elliptic curves} for non-CM elliptic curves; observe that $\mathbf{H}^\times/\RR_{>0}$ is naturally isomorphic to $\mathbf{H}_1$ where $\RR_{>0}$ is the group of positive real numbers.

\begin{remark}
We have used Dirichlet density instead of natural density because the finite extensions of $L$ arising from $\rho_\infty$ are not geometric, i.e., the field of constants will grow.  Natural density can be used if one keeps in mind that $\rho_{\infty}(\Gal(L^\sep/L\kbar))=\rho_{\infty}(W_L)\cap \OO_D^\times$.
\end{remark}

There are many possibilities for the image of $\rho_\infty$ and hence there are many possible Sato-Tate laws for a Drinfeld module $\phi$; this contrasts with elliptic curves where there are only two expected Sato-Tate laws.  \\    

To give a concrete description of an equidistribution law, we now focus on a special case: the distribution of traces of Frobenius  when $\rho_\infty$ is surjective.   

For each closed point $x$ of $X$, we define the \defi{degree} of $x$ to be $\deg(x)=[\FF_x:\FF_\infty]$.   For each integer $d\geq 1$, let $|X|_d$ be the set of degree $d$ closed points of $X$.  Note that $|X|_d$ is empty if $d$ is not divisible by $[\FF_L:\FF_\infty]$ where $\FF_L$ is the field of constants of $L$.  (This notion of degree depends not only on $X$ but on the extension $L/F$ arising from $\phi$.)
  
 For each closed point $x$ of $X$, let $a_x(\phi)\in A$ be the \defi{trace of Frobenius} of $\phi$ at $x$; it is $(-1)^{n-1}$ times the coefficient of $T^{n-1}$ in $P_{\phi,x}(T)$.     We have $a_x(\phi)=\tr(\rho_\infty(\Frob_x))$ where $\tr\colon D\to F_\infty$ is the reduced trace map.  The Drinfeld module analogue of the Hasse bound says that $\ord_\infty(a_x(\phi)) \geq -\deg(x)/n$, and hence $a_x(\phi) \pi^{\lfloor \deg(x)/n \rfloor}$ belongs to $\OO_\infty$ where $\pi$ is a uniformizer of $F_\infty$.

\begin{theorem} \label{T:ST for surjective}
Let $\phi\colon A \to L[\tau]$ be a Drinfeld module of rank $n\geq 2$ with generic characteristic where $L$ is finitely generated.   Assume that $\rho_\infty(W_L)=D_\phi^\times$.    

Let $\pi$ be a uniformizer for $F_\infty$ and let $\mu$ be the Haar measure of $\OO_\infty$ normalized so that $\mu(\OO_\infty)=1$.  Let $\mathcal{S}$ be the set of positive integers that are divisible by $[\FF_L:\FF_\infty]$.  Fix a scheme $X$ as in \S\ref{SS:Sato-Tate law}.
\begin{romanenum}
\item  \label{I:ST for surjective i}
For an open subset $U$ of $\OO_\infty$, we have
\[
\lim_{\substack{d\in \mathcal{S},\, d\not\equiv 0 \,(\!\bmod{n}) \\d\to +\infty}} \frac{\#\{ x \in |X|_d :  a_x(\phi) \pi^{\lfloor d/n \rfloor} \in U \}}{\#|X|_d} = \mu(U).
\]
\item
Let $\nu$ be the measure on $\OO_\infty$ such that if $U$ is an open subset of one of the cosets $a+\pi \OO_\infty$ of $\OO_\infty$, then
\[
\nu(U) = \begin{cases}
      (q^{d_\infty (n-1)} -1)/ (q^{d_\infty n}-1)\cdot \mu(U) & \text{ if $U\subseteq \pi\OO_\infty$}, \\
      \,\,q^{d_\infty (n-1)}/ (q^{d_\infty n}-1)\cdot \mu(U) & \text{ otherwise}.
\end{cases}
\]
For an open subset $U$ of $\OO_\infty$, we have
\[
\lim_{\substack{d\in \mathcal{S},\, d\equiv 0 \,(\!\bmod{n}) \\d\to +\infty}} \frac{\#\{ x \in |X|_d :  a_x(\phi) \pi^{\lfloor d/n \rfloor} \in U \}}{\#|X|_d} = \nu(U).
\]
\end{romanenum}
\end{theorem}

\begin{remark}
Theorem~\ref{T:ST for surjective}(\ref{I:ST for surjective i}) proves much of a conjecture of E.-U.~Gekeler \cite{MR2366959}*{Conjecture~8.18}; which deals with rank 2 Drinfeld modules over $L=F=k(t)$ with $\pi=t^{-1}$.  (Gekeler's assumptions are weaker than $\End_{\Lbar}(\phi)=\phi(A)$ with $\rho_\infty$ surjective).    
\end{remark}

 \subsection{Application: Lang-Trotter bounds}
  Let $\phi\colon A \to L[\tau]$ be a Drinfeld module of rank $n$ with generic characteristic.  For simplicity, we assume that $L$ is a global function field and that $\End_{\Lbar}(\phi)=\phi(A)$.  Fix $X$ as in \S\ref{SS:Sato-Tate law}. 

Fix a value $a\in A$, and let $P_{\phi,a}(d)$ be the number of closed points $x$ of $X$ of degree $d$ such that $a_x(\phi)=a$ (see the previous section for definitions).   We will prove the following bound for $P_{\phi,a}(d)$ with our Sato-Tate law.

\begin{theorem} \label{T:LT bound}
With assumption as above, we have 
\[
P_{\phi,a}(d) \ll q^{d_\infty(1-1/n^2)d}
\] where the implicit constant depends only on $\phi$ and $\ord_\infty(a)$.
\end{theorem}

The most studied case is $n=2$ which is analogous to the case of non-CM elliptic curves (see Remark~\ref{R:L-T}).  With $F=k(t)$, $A=k[t]$ and $L=F$, A.C.~Cojocaru and C.~David have shown that $P_{\phi,a}(d) \ll q^{(4/5)d}/d^{1/5}$ and $P_{\phi,0}(d) \ll q^{(3/4)d}$ where the implicit constant does not depend on $a$ (this also can be proved with the Sato-Tate law).  For $n=2$, the above theorem gives $P_{\phi,a}(d) \ll q^{(3/4)d}$ for all $a$.   For arbitrary rank $n\geq 2$, David \cite{MR1858082} proved that $P_{\phi,a}(d)\ll q^{\theta(n) d}/d$ where $\theta(n):=1- 1/(2n^2+4n)$.   These bounds were proved using the $\lambda$-adic representations ($\lambda\neq \infty$) associated to $\phi$.

\begin{remark} \label{R:L-T}
Let $E$ be a non-CM elliptic curve over $\QQ$.  Fix an integer $a$, and let $P_{E,a}(x)$ be the number of primes $p\leq x$ for which $E$ has good reduction and $a_p(E)=a$.  The \defi{Lang-Trotter conjecture} says that there is a constant $C_{E,a}\geq 0$ such that $P_{E,a}(x)\sim C_{E,a}\cdot x^{1/2}/\log x$ as $x\to +\infty$; see \cite{MR0568299} for heuristics and a description of the conjectural constant (if $C_{E,a}=0$, then the asymptotic is defined to mean that $P_{E,a}(x)$ is bounded as a function of $x$).   Under GRH, Murty, Murty and Saradha showed that $P_{E,a}(x)\ll x^{4/5}/(\log x)^{1/5}$ for $a\neq 0$ and $P_{E,0}(x)\ll x^{3/4}$ \cite{MR935007}.

Assuming a very strong form of the Sato-Tate conjecture for $E$ (i.e., the $L$-series attached to symmetric powers of $E$ have analytic continuation, functional equation and satisfy the Riemann hypothesis), V.~K.~Murty showed that $P_{E,a}(x)\ll x^{3/4} (\log x)^{1/2}$, see \cite{MR823264}.
\end{remark}

Let $|X|_d$ be the set of closed points of $X$ with degree $d$.   We shall assume from now on that $d$ is a positive integer divisible by $[\FF_L:k]$ where $\FF_L$ is the field of constants in $L$ (otherwise, $|X|_d=\emptyset$ and $P_{\phi,a}(d)=0$).   

Let us give a crude heuristic for an upper bound of $P_{\phi,a}(d)$.  Fix a point $x\in|X|_d$.  By the Drinfeld module analogue of the Hasse bound, we have $\ord_\infty(a_x(\phi)) \geq  - d/n$.  The Riemann-Roch theorem then implies that $|\{ f \in A : \ord_\infty(f) \geq - d/n \}| = q^{\lfloor d/n \rfloor d_\infty + 1 - g}$ for all sufficiently large $d$, where $g$ is the genus of $F$.   So assuming $a_x(\phi)$ is a ``random'' element of the set $\{f\in A: \ord_\infty(f) \geq -d/n \}$, we find that the ``probability'' that $a_x(\phi)$ equals $a$ is  $O(1/q^{d_\infty\cdot d/n})$.    So we conjecture that
\[
P_{\phi,a}(d) \ll \sum_{x\in |X|_d} \frac{1}{q^{d_\infty\cdot d/n}}= \#|X|_d \cdot \frac{1}{q^{d_\infty\cdot d/n}} \ll  \frac{q^{d_\infty\cdot d}}{d} \frac{1}{q^{d_\infty\cdot d/n}} = \frac{q^{d_\infty (1-1/n)d}}{d}.
\]

\begin{remark}
In this paper, we are only interested in upper bounds.  The most optimistic analogue of the Lang-Trotter conjecture would be the following: there is a positive integer $N$ and constants $C_{\phi,a}(d)\geq 0$ such that 
\[
P_{\phi,a}(d) \sim C_{\phi,a}(d) \cdot {q^{d_\infty (1-1/n)d}}/{d}
\] 
as $d\to +\infty$ where $C_{\phi,a}(d)$ depends only on $\phi$, $a$ and $d$ modulo $N$.   The Sato-Tate conjecture for $\phi$ would be an ingredient for an explicit description of the constant $C_{\phi,a}(d)$.  (The conjecture is in general false if we insist that $N=1$.  For rank 2 Drinfeld modules over $k(t)$ and $a=0$, \cite{MR1373559}*{Theorem~1.2} suggests that $N$ is usually 2.)
\end{remark}

To prove Theorem~\ref{T:LT bound}, we will consider the image of $\rho_\infty$ in the quotient $D_\phi^\times/(F_\infty^\times(1+\pi^j \OO_D))$ where $\pi$ is a uniformizer of $F_\infty$ and $j\approx d/n^2$.

\subsection{Compatible system of representations}

Let $\phi\colon A \to L[\tau]$ be a Drinfeld module of rank $n$.  For a non-zero ideal $\aA$ of $A$, let $\phi[\aA]$ be the group of $b\in \Lbar$ such that $\phi_a(b)=0$ for all $a\in A$ (where we identify each $\phi_a$ with the corresponding polynomial in $L[X]$).  The group $\phi[\aA]$ is an $A/\aA$-module via $\phi$ and if $\aA$ is not divisible by the characteristic of $\phi$, then $\phi[\aA]$ is a free $A/\aA$-module of rank $n$.  For a fixed place $\lambda\neq \infty$ of $F$, let $\p_\lambda$ be the maximal ideal of $\OO_\lambda$.  The $\lambda$-adic \defi{Tate module} of $\phi$ is defined to be 
\[
T_\lambda(\phi) := \Hom_{A_\lambda}\Bigl(F_\lambda/\OO_\lambda , \, \varinjlim_i \phi[\p_\lambda^i] \Bigr).
\]
If $\p_\lambda$ is not the characteristic of $\phi$, then $T_\lambda(\phi)$ is a free $\OO_\lambda$-module of rank $n$.  There is a natural Galois action on $T_\lambda(\phi)$ which gives a continuous homomorphism
\[
\rho_\lambda\colon \Gal_L \to \Aut_{\OO_\lambda}(T_\lambda(\phi)). 
\]

Now suppose that $\phi$ has generic characteristic and that $L$ is a finitely generated.  Take a scheme $X$ as in \S\ref{SS:Sato-Tate law}.   For a closed point $x$ of $X$, let $\lambda_x$ be the place of $F$ corresponding to the characteristic of $\phi_x$.  For a place $\lambda\neq \lambda_x$ of $F$, we have
\[
P_{\phi,x}(T)=\det(TI - \rho_\lambda(\Frob_x) )
\]
(for $\lambda\neq \infty$, we are using $ \Aut_{\OO_\lambda}(T_\lambda(\phi))\cong \GL_n(\OO_\lambda)$ and \cite{MR1196527}*{Theorem 3.2.3(b)}).  This property is one of the reasons it makes sense to view $\rho_\infty$ as a member of the family of compatible representations $\{\rho_\lambda\}$.

There is a natural map $\End_\Lbar(\phi) \hookrightarrow \End_{\OO_\lambda}(T_\lambda(\phi))$ and the image of $\rho_\lambda$ commutes with $\End_L(\phi)$.   We can now state the following important theorem of R.~Pink; it follows from \cite{MR1474696}*{Theorem~0.2} which is an analogue of Serre's open image theorem for elliptic curves \cite{MR0387283}.    Theorem~\ref{T:main theorem} can thus be viewed as the analogue of this theorem for the place $\infty$; our proof will closely follow Pink's.

\begin{theorem}[Pink] \label{T:Pink MT}
\label{T:Pink MT}
Let $\phi\colon A\to L[\tau]$ be a Drinfeld module with generic characteristic, and assume that the field $L$ is finitely generated.  Then for any place $\lambda\neq \infty$ of $F$,  the image of
\[
\rho_\lambda\colon \Gal_L \to \Aut_{\OO_\lambda}(T_\lambda(\phi))
\]
is commensurable with $\Cent_{\End_{\OO_\lambda}(T_\lambda(\phi))}(\End_\Lbar(\phi))^\times$.
\end{theorem}

\begin{example} [Explicit class field theory for rational function fields]  \label{Ex:CFT Hayes}

As an example of the utility of viewing $\rho_\infty$ as a legitimate member of the family $\{\rho_\lambda\}_\lambda$, we give an explicit description of the maximal abelian extension $F^\ab$ in $F^\sep$ of the field $F=k(t)$, where $k$ is a finite field with $q$ elements.  We will recover the description of $F^\ab$ of Hayes in \cite{MR0330106}.   Using the ideas arising from this paper, we have given a description of $F^{\ab}$ for a general global function field $F$, see \cite{Zywina-CFT}.

Let $\infty$ be the place of $F$ correspond to the valuation for which $\ord_\infty(f)=-\deg f(t)$ for each non-zero $f\in k[t]$; the function $t^{-1}$ is a uniformizer for $\OO_\infty$.  The ring of rational functions that are regular away from $\infty$ is $A=k[t]$.    Let $\phi\colon A \to F[\tau]$ be the homomorphism of $k$-algebras that satisfies $\phi_t=t + \tau$; this is a Drinfeld module of rank $1$ called the \defi{Carlitz module}.   

If $\p$ is a \emph{monic} irreducible polynomial of $k[t]$, then $\rho_\lambda(\Frob_\p)=\p$ for every place $\lambda$ of $F$ except for the one corresponding to $\p$ (for $\lambda\neq \infty$, this follows from \cite{MR0330106}*{Cor.~2.5}).  In particular, one finds that the image of $\rho_\infty\colon W_F\to D_\phi^\times=F_\infty^\times$ must lie in $\ang{t}\cdot (1+t^{-1}\OO_\infty)$.    For $\lambda\neq \infty$, we make the identification $\Aut_{\OO_\lambda}(T_\lambda(\phi))=\OO_\lambda^\times$.  Combining our $\lambda$-adic representations together, we obtain a single continuous homomorphism
\[
\prod_{\lambda} \rho_\lambda \colon W_F^{\ab} \to \Big(\prod_{\lambda\neq \infty} \OO_\lambda^\times\Big) \times \ang{t}\cdot (1+t^{-1}\OO_\infty).
\]

Let $\mathbf{A}^\times_F$ be the idele group of $F$.  The homomorphism $(\prod_{\lambda\neq \infty} \OO_\lambda^\times) \times \ang{t}\cdot (1+t^{-1}\OO_\infty)\to \mathbf{A}^\times_F/F^\times$  obtained by composing the inclusion into $\mathbf{A}_F^\times$ with the quotient map is an isomorphism.   Composing ${\prod}_{\lambda} \rho_\lambda$ with this map, we obtain a continuous homomorphism
\[
\beta\colon W_F^{\ab}\to \mathbf{A}^\times_F/F^\times.
\]
The map $\beta$ embodies explicit class field theory for $F$.  Indeed, it is an isomorphism and the homomorphism $W_F^{\ab}\xrightarrow{\sim} \mathbf{A}^\times_F/F^\times,\, s \mapsto \beta(s^{-1})$ is the inverse of the \defi{Artin map} of class field theory!  See Remark~\ref{R:CFT Hayes}, for further details.  In particular, observe that the homomorphism $\beta$ does not depend on our choice of $\infty$ and $\phi$.   

By taking profinite completions, we obtain an isomorphism
\[
\Gal(F^\ab/F) \xrightarrow{\sim}  \Big(\prod_{\lambda\neq \infty} \OO_\lambda^\times\Big) \times \widehat{\ang{t}}\cdot (1+t^{-1}\OO_\infty).
\]
of profinite groups.  This isomorphism allows us to view $F^\ab$ as the compositum of three linearly disjoint fields.   The first is the union $K_1$ of the fields $F(\phi[\aA])$ where $\aA$ varies over the non-zero ideals of $A$, see \cite{MR0330106} for details; these extensions were first constructed by Carlitz.  We have $\Gal(K_1/F)\cong \prod_{\lambda\neq \infty} \OO_\lambda^\times$.   The second extension is the the field $K_2=\kbar(t)$; it satisfies $\Gal(K_2/F)\cong \Gal(\kbar/k) \cong \widehat{\ZZ}$.    

Finally, let us describe the third field $K_3\subset F^\ab$, i.e., the subfield for which $\rho_\infty$ induces an isomorphism $\Gal(K_3/F)\xrightarrow{\sim} 1+t^{-1}\OO_\infty$.    We first find a series $u=\sum_{i=0}^\infty a_i \tau^{-i} \in \Fbar\twistedseries^\times$ for which $u^{-1}\phi_t u = \tau$, and hence $u^{-1}\phi(A)u\subseteq \kbar\twistedLaurent$.  Expanding out $\phi_t u = u\tau$, this translates into the equations:
\[
a_0\in k^\times\quad\text{ and }\quad a_{j+1}^q-a_{j+1}= -t a_j \quad \text{ for $j\geq0$}.
\]
Set $a_0=1$ and recursively find  $a_j \in F^\sep$ that satisfy these equations.   We then have a chain of fields $F=F(a_0)\subseteq F(a_1)\subseteq F(a_2) \subseteq \ldots$.  Note that the field $F(a_j)$ does not depend on the choice of $a_j$ and $[F(a_j):F] \leq q^j$.   For each $j\geq 0$, let $L_j$ be the subfield of $K_3$ for which $\rho_\infty$ induces an isomorphism $\Gal(L_j/F) \xrightarrow{\sim} (1+t^{-1}\OO_\infty)/(1+t^{-(j+1)}\OO_\infty)$.    The field $L_j$ depends only on $u \pmod{\tau^{-(j+1)}\Fbar\twistedseries}$, so we have $L_j \subseteq F(a_j)$.   Since $q^j=[L_j:F]\leq [F(a_j):F]\leq q^j$, we deduce that 
\[
K_3=\bigcup_{j\geq 0} F(a_j) \quad \text{and}\quad \Gal(F(a_j)/F)\cong (1+t^{-1}\OO_\infty)/(1+t^{-(j+1)}\OO_\infty).
\]
In \cite{MR0330106}, Hayes constructs the three fields $K_1,K_2,K_3$ and then showed that their compositum is $F^\ab$.   The field $K_3$ is constructed by consider the torsion points of another Drinfeld module but starting with the ring $k[t^{-1}]$ instead.    The advantage of including $\rho_\infty$ is that the proof is easier and that the fields $K_2$ and $K_3$ arise naturally from our canonical map $\beta$.
\end{example}

\subsection{Overview}

In \S\ref{S:construction}, we shall define our Sato-Tate representation $\rho_\infty$ and prove its basic properties.  

In \S\ref{SS:rank 1}, we prove the rank 1 case of Theorem~\ref{T:main theorem}.  The proof essentially boils down to an application of class field theory.  The rank 1 case will also be a key ingredient in the general proof of Theorem~\ref{T:main theorem}.  

In \S\ref{S:Tate}, we shall prove an $\infty$-adic version of the Tate conjecture.   The prove entails replacing $\phi$ with its associated $A$-motive (though we will not use that terminology), and then using Tamagawa's analogue of the Tate conjecture.  We have avoided the temptation to define a Sato-Tate law for general $A$-motives, but the author hopes to return to it in future work (the corresponding openness theorem would likely be extremely difficult for a general $A$-motive since the general version of Theorem~\ref{T:Pink MT} remains open).

In \S\ref{S:proof of non-CM}, we prove Theorem~\ref{T:main theorem non-CM}.   Our proof uses most of the ingredients from Pink's proof of Theorem~\ref{T:Pink MT}.  In \S\ref{S:main theorem proof}, we deduce Theorem~\ref{T:main theorem} from Theorem~\ref{T:main theorem non-CM}.

Finally, in \S\ref{S:equidistribution proof} and \S\ref{S:LT proof}, we prove Theorem~\ref{T:ST for surjective} and Theorem~\ref{T:LT bound}, respectively.   A key ingredient in both proofs is the Chebotarev density theorem.

\subsection*{Acknowledgements} Thanks to Bjorn Poonen and Lenny Taelman.

\section{Construction of $\rho_\infty$} \label{S:construction}
 Let $\phi\colon A\to L[\tau]$ be a Drinfeld module.  As noted in \S\ref{SS:Sato-Tate law}, $\phi$ extends uniquely to a homomorphism 
\[
\phi\colon F_\infty \hookrightarrow L^\perf\twistedLaurent
\]
that satisfies (\ref{E:degrees}) for all non-zero $x\in F_\infty$.   

Our first task is to prove that there exists a series $u\in \Lbar\twistedLaurent^\times$ for which $u^{-1} \phi(F_\infty) u \subseteq \kbar\twistedLaurent$; this is shown in \cite{MR2018826}*{\S2}, but we will reprove it in order to observe that the coefficients of $u$ actually lie in $L^\sep$.  Fix a non-constant $y\in A$.   We have $\phi_y = \sum_{j=0}^{h} b_j \tau^{j}$ with $b_j\in L$ and $b_h\neq 0$, where $h:=-nd_\infty\ord_\infty(y)$.    Choose a solution $\delta\in L^\sep$ of $\delta^{q^h-1}= 1/b_h$.    Set $a_0=1$ and recursively solve for $a_1,a_2,a_3\ldots \in \Lbar$ by the equation
\begin{equation} \label{E:artin-schreier}
a_i^{q^h} - a_i = - \sum_{\substack{0\leq j\leq h-1\\ i+j-h\geq 0}}  \delta^{q^{j}-1} b_j a_{i+j-h}^{q^j}.
\end{equation}
The $a_i$ belong to $L^\sep$ since (\ref{E:artin-schreier}) is a separable polynomial in $a_i$ and the values $b_j$ and $\delta$ belong to $L^\sep$.

\begin{lemma} \label{L:u exists}
With $\delta$ and $a_i$ as above, the series $u:= \delta (\sum_{i=0}^\infty a_i \tau^{-i}) \in \Lbar\twistedLaurent^\times$ has coefficients in $L^\sep$ and satisfies $u^{-1} \phi(A) u \subseteq \kbar\twistedLaurent$.
\end{lemma}
\begin{proof}
 Expanding out the series $\phi_y u$ and $u\tau^{h}$ and comparing, we find that $\phi_y u=u\tau^{h}$ (use (\ref{E:artin-schreier}) and $\delta^{q^h-1}= 1/b_h$).  Let $k_h$ be the degree $h$ extension of $k$ in $\kbar$.  The elements of the ring $\Lbar\twistedLaurent$ that commute with $\tau^h$ are $k_h\twistedLaurent$.  Since $\tau^h=u^{-1}\phi_y u$ belongs to the commutative ring $u^{-1}\phi(F_\infty) u$, we find that $u^{-1}\phi(F_\infty) u$ is a subset of $k_h\twistedLaurent$.   Thus $u \in \Lbar\twistedLaurent^\times$ has coefficients in $L^\sep$ and satisfies $u^{-1} \phi(F_\infty) u \subseteq \kbar\twistedLaurent$.
\end{proof}

Choose any series $u\in \Lbar\twistedLaurent^\times$ that satisfies $u^{-1}\phi(A)u \subseteq \kbar\twistedLaurent$ and has coefficients in $L^\sep$.  Define the function
\begin{align*}
\rho_\infty\colon W_L &\to D_\phi^\times,\quad \sigma\mapsto \sigma(u) \tau^{\deg(\sigma)} u^{-1}.
\end{align*}
Recall that $D_\phi$ is the centralizer of $\phi(A)$ in $\Lbar\twistedLaurent$.   The following lemma gives some basic properties of $\rho_\infty$; we will give the proof in \S\ref{SS:facts proofs}.   In particular, $\rho_\infty$ is a well-defined continuous homomorphism that does not depend on the initial choice of $u$.  Our construction varies slightly from Yu's, cf.~\S\ref{SS:difference in construction}.

\begin{lemma} \label{L:u facts}
\begin{romanenum}
\item \label{I:u fact i}
There is a series $u \in \Lbar\twistedLaurent^\times$ that satisfies $u^{-1}\phi(F_\infty) u \subseteq \kbar\twistedLaurent$, and any such $u$ has coefficients in $L^\sep$.
\item \label{I:u fact ii}  
The ring $D_\phi$ is a central $F_\infty$-division algebra with invariant $-1/n$. 
\item \label{I:u fact iii}
Fix $u$ as in (\ref{I:u fact i}) and take any $\sigma \in W_L$.  The series $\sigma(u)\tau^{\deg(\sigma)} u^{-1}$ belongs to $D_\phi^\times$ and does not depend on the initial choice of $u$.
\item \label{I:u fact iv}
For $\sigma\in W_L$, we have $\ord_{\tau^{-1}} \rho_\infty(\sigma)=-\deg(\sigma)$.
\item 
The map $\rho_\infty\colon W_L \to D_\phi^\times$ is a continuous group homomorphism.   
\item  
The group $\rho_\infty(W_L)$ commutes with $\End_L(\phi)$.
\end{romanenum}
\end{lemma}

\begin{lemma} \label{L:infty independence}
Assume that $\phi$ has generic characteristic, $L$ is finitely generated, and let $X$ be a scheme as in \S\ref{SS:Sato-Tate law}.  Then the homomorphism $\rho_\infty\colon W_L \to D_\phi^\times$ is unramified at each closed point of $x$ of $X$ and we have $P_{\phi,x}(T)=\det(TI-\rho_\infty(\Frob_x))$.
\end{lemma}
\begin{proof}
These results follow from \cite{MR2018826}; they only depend on $\rho_\infty$ up to conjugacy so we may use Yu's construction (see \S\ref{SS:difference in construction}).  Note that Lemma~{3.2} of \cite{MR2018826} should use the arithmetic Frobenius instead of the geometric one; the contents of that lemma have been reproved below in Example~\ref{Ex:finite Drinfeld}.
\end{proof}

\begin{example} \label{Ex:finite Drinfeld}
Let $\phi\colon A\to L[\tau]$ be a Drinfeld module of rank $n$ and $L$ a finite field.   The group $W_L$ is cyclic and generated by the automorphism $\Frob_L \colon x\mapsto x^{|L|}$.  We have $L^\sep=\kbar$, and hence $u:=1$ satisfies the condition of Lemma~\ref{L:u facts}(\ref{I:u fact i}).   Thus $\rho_\infty(\sigma) = \sigma(u)\tau^{\deg \sigma}u^{-1}=\tau^{\deg(\sigma)}$ for all $\sigma\in W_L$, and in particular, $\rho_{\infty}(\Frob_L)=\tau^{[L:k]}$.   Note that $\pi:=\tau^{[L:k]}$ belongs to $\End_L(\phi)$.  

Let $E$ be the subfield of $\End_{L}(\phi)\otimes_A F$ generated by $F$ and $\pi$.  Let $f_\phi \in F[T]$ be the minimal polynomial of $\pi$ over $F$.  The \defi{characteristic polynomial} $P_\phi(T)$ of $\pi$ is the degree $n$ polynomial that is a power of $f_\phi(T)$.   

 By \cite{MR0439758}*{Prop.~2.1}, $E\otimes_F F_\infty$ is a field and hence $f_\phi$ is also the minimal polynomial of $\pi$ over $F_\infty$.  The characteristic polynomial of the $F_\infty$-linear map $D_\phi \to D_\phi,\, a\mapsto \pi a$ is thus a power of $f_\phi$.    This implies that the degree $n$ polynomial $\det(TI - \pi)$ is a power of $f_\phi$, and hence equals $P_\phi(T)$.  Therefore,
\[
P_\phi(T)=\det(TI - \rho_{\infty}(\Frob_L)).
\]
\end{example}

\subsection{Proof of Lemma~\ref{L:u facts}} \label{SS:facts proofs}
Fix a uniformizer $\pi$ of $F_\infty$.  There is a unique embedding $\iota\colon F_\infty \to \kbar\twistedLaurent$ of rings that satisfies the following conditions:
\begin{itemize}
\item
$\iota(x)=x$ for all $x\in \FF_\infty$,
\item $\iota(\pi)=\tau^{-nd_\infty}$,
\item $\ord_{\tau^{-1}}(\iota(x))=nd_\infty \ord_\infty(x)$ for all $x\in F_\infty^\times$.
\end{itemize}
Let $k_{d_\infty}$ and $k_{nd_\infty}$ be the degree $d_\infty$ and $nd_\infty$ extensions of $k$, respectively, in $\kbar$.  We have $\iota(F_\infty)=k_{d_\infty}(\!( \tau^{-nd_\infty})\!)$.   Let $D_\iota$ be the centralizer of $\iota(F_\infty)$ in $\Lbar\twistedLaurent$; it is an $F_\infty$-algebra via $\iota$.   Using that $k_{d_\infty}$ and $\tau^{nd_\infty}$ are in $\iota(F_\infty)$, we find that $D_\iota=k_{nd_\infty}(\!(\tau^{-d_\infty})\!)$.   One can verify that $D_\iota$ is a central $F_\infty$-division algebra with invariant $-1/n$.  \\

By Lemma~\ref{L:u exists}, there is a series $u \in \Lbar\twistedLaurent^\times$ with coefficients in $L^\sep$ such that $u^{-1}\phi(F_\infty) u \subseteq \kbar\twistedLaurent$.   Take any $v\in \Lbar\twistedLaurent^\times$ that satisfies $v^{-1} \phi(F_\infty) v \subseteq \kbar\twistedLaurent$.  By \cite{MR2018826}*{Lemma~2.3}, there exist $w_1$ and $w_2\in \kbar\twistedseries^\times$ such that
\[
\iota(x) = w_1^{-1}(u^{-1} \phi_x u) w_1 \quad \text{ and }\quad \iota(x) = w_2^{-1}(v^{-1} \phi_x v) w_2
\]
for all $x\in F_\infty$.  So for all $x\in F_\infty$, we have $(uw_1)  \iota(x) (uw_1)^{-1} =\phi_x = (v w_2)  \iota(x) (v w_2)^{-1}$ and hence
\[
(w_2^{-1} v^{-1} u w_1) \iota(x) (w_2^{-1} v^{-1} u w_1)^{-1} = \iota(x).
\]
Therefore $w_2^{-1} v^{-1} u w_1$ belongs to $D_\iota\subseteq \kbar\twistedLaurent$, and hence $v=uw$ for some $w\in \kbar\twistedLaurent^\times$.  The coefficients of $v$ lie in $L^\sep$ since the coefficients of $u$ lie in $L^\sep$ and $w$ has coefficients in the perfect field $\kbar\subseteq L^\sep$.  This completes the proof of (i). 

We have shown that the series $g:= uw_1 \in \Lbar\twistedLaurent$ satisfies $\iota(x)=g^{-1} \phi_x g$ for all $x\in F_\infty$.  The map $D_\phi \to D_\iota, \, f \mapsto g^{-1} f g$ is an isomorphism of $F_\infty$-algebras.   Therefore, $D_\phi$ is also a central $F_\infty$-division algebra with invariant $-1/n$; this proves (ii).

Take any $\sigma\in W_L$.   Since $w$ has coefficients in $\kbar$, we have $\sigma(w)=\tau^{\deg(\sigma)} w \tau^{-\deg(\sigma)}$ and hence
\begin{align*}
\sigma(v) \tau^{\deg(\sigma)} v^{-1} &= \sigma(uw)\tau^{\deg(\sigma)} (uw)^{-1}\\
&= \sigma(u) \sigma(w) \tau^{\deg(\sigma)} w^{-1} u^{-1}\\
& = \sigma(u) (\tau^{\deg(\sigma)} w \tau^{-{\deg(\sigma)}})\tau^{\deg(\sigma)} w^{-1} u^{-1}\\
&=\sigma(u) \tau^{\deg(\sigma)} u^{-1}.
\end{align*}
This proves that $\rho_\infty(\sigma):=\sigma(u) \tau^{\deg(\sigma)} u^{-1}$ is independent of the initial choice of $u$.  

To complete the proof of (iii), we need only show that $\rho_\infty(\sigma)$ commutes with $\phi(A)$.   We will now prove (vi), which says that $\rho_\infty(\sigma)$ commutes with the even larger ring $\End_L(\phi)$.   Take any non-zero $f\in \End_L(\phi)$.  Since $f$ commutes with $\phi(A)$, and hence with $\phi(F_\infty)$, we have $(fu)^{-1}\phi(F_\infty)(fu) \subseteq \kbar\twistedLaurent$.   Since $\rho_\infty(\sigma)$ does not depend on the choice of $u$, we have
\[
\rho_\infty(\sigma) = \sigma(fu) \tau^{\deg(\sigma)} (fu)^{-1} = \sigma(f)\sigma(u) \tau^{\deg(\sigma)} u^{-1} f^{-1} = \sigma(f) \rho_\infty(\sigma) f^{-1}.
\]
Since $f$ has coefficients in $L$, we deduce that $\rho_\infty(\sigma) f = f \rho_\infty(\sigma)$, as desired.

For part (iv), note that
\begin{align*}
\ord_{\tau^{-1}}(\rho_\infty(\sigma)) &= \ord_{\tau^{-1}}(\sigma(u)) + \ord_{\tau^{-1}}(\tau^{\deg(\sigma)}) - \ord_{\tau^{-1}}(u)\\
&= \ord_{\tau^{-1}}(u) -\deg(\sigma) - \ord_{\tau^{-1}}(u)=-\deg(\sigma).
\end{align*}

It remains to prove part (v).  We first show that $\rho_\infty$ is a group homomorphism.  For $\sigma_1,\sigma_2\in W_L$, we have
\begin{align*}
\rho_\infty(\sigma_1 \sigma_2) = (\sigma_1\sigma_2)(u) \tau^{\deg(\sigma_1\sigma_2)}  u^{-1}= \sigma_1(\sigma_2(u)) \tau^{\deg(\sigma_1)} \sigma_2(u)^{-1}\cdot \sigma_2(u) \tau^{\deg(\sigma_2)}u^{-1}  = \rho_\infty(\sigma_1)\rho_\infty(\sigma_2).
\end{align*}  
We have used part (iii) along with the observation that if $u^{-1}\phi(A)u \subseteq \kbar\twistedLaurent$, then $\sigma_2(u)^{-1}\phi(A)\sigma_2(u) \subseteq \kbar\twistedLaurent$.   

Finally, we prove that $\rho_\infty$ is continuous.  By Lemma~\ref{L:u exists}, we may assume that $u$ is of the form $\sum_{i=0}^\infty a_i \tau^{-i}$ with $a_0\neq 0$.  Let $\OO_{D_\phi}$ be the valuation ring of $\ord_{\tau^{-1}}\colon D_\phi \to \ZZ\cup \{+\infty\}$; it is a local ring.  By part (iv), we need only show that the homomorphism $\Gal(L^\sep/L\kbar) \xrightarrow{\rho_\infty} \OO_{D_\phi}^\times$ is continuous.  It thus suffices to show that for each $j\geq 1$, the homomorphism
\[
\beta_j \colon  \Gal(L^\sep/L\kbar) \xrightarrow{\rho_\infty} \OO_{D_\phi}^\times \to (\OO_{D_\phi} / \pi^j \OO_{D_\phi})^\times
\]
has open kernel, where $\pi$ is a fixed uniformizer of $F_\infty$.  For each $\sigma \in  \Gal(L^\sep/L\kbar)$, we have $\rho_\infty(\sigma)= \sigma(u) u^{-1}$.  One can check that $\beta_j(\sigma)=1$, equivalently $\ord_{\tau^{-1}}(\rho_\infty(\sigma)-1)\geq \ord_{\tau^{-1}}(\phi_\pi^j)=nd_\infty j$, if and only if $\ord_{\tau^{-1}}(\sigma(u) u^{-1} - 1)=\ord_{\tau^{-1}}(\sigma(u)-u)$ is at least $nd_\infty j$.   Thus the kernel of $\beta_j$ is $\Gal(F^\sep/L_j)$ where $L_j$ is the finite extension of $L\kbar$ generated by the set $\{a_i\}_{0\leq i< nd_\infty j}$.

\subsection{Yu's construction} \label{SS:difference in construction}
Let us relate our representation $\rho_\infty$ to that given by J.K.~Yu in \cite{MR2018826}.  Assume that $L$ is perfect.  Let $\iota\colon F_\infty\to \kbar\twistedLaurent$ be the embedding of \S\ref{SS:facts proofs}.  Choose a series $u_0\in\Lbar\twistedLaurent^\times$ for which $\iota(x)=u_0\phi_xu_0^{-1}$ for all $x\in F_\infty$.  The representation defined in \cite{MR2018826}*{\S2.5} is
\[
\varrho_\infty\colon W_L \to D_\iota^\times,\quad \sigma\mapsto u_0 \sigma(u_0)^{-1}\tau^{\deg(\sigma)}
\]
where $D_\iota$ is the central $F_\infty$-division algebra with invariant $-1/n$ described at the beginning of \S\ref{SS:facts proofs}.  The connection with our representation is that
\[
\rho_\infty(\sigma)= \sigma(u_0^{-1}) \tau^{\deg(\sigma)} (u_0^{-1})^{-1} = \sigma(u_0)^{-1} \tau^{\deg(\sigma)} u_0 = u_0^{-1} \varrho_\infty(\sigma) u_0
\]
for all $\sigma\in W_L$.   A different choice of $u_0$ will change $\varrho_\infty$ by an inner automorphism of $D_\iota^\times$.  (For $L$ not perfect, one can construct $\rho_\infty\colon W_{L^\perf}\to D_\iota^\times$ as above, and then use the natural isomorphism $W_L=W_{L^\perf}$.)

\subsection{Aside: Formal modules}
Let us quickly express the above construction in terms of \emph{formal modules}; this will not be needed elsewhere.  Let $\phi\colon A \to L[\tau]$ be a Drinfeld module of rank $n$ and assume that $L$ is perfect.  Then $\phi$ extends uniquely to a homomorphism $\phi\colon F_\infty \hookrightarrow L\twistedLaurent$ that satisfies (\ref{E:degrees}) for all non-zero $x\in F_\infty$.    In particular, restricting to $\OO_\infty$ defines a homomorphism $\OO_\infty \to L\twistedseries$.

To each formal sum $f=\sum_{i\in \ZZ} a_i \tau^{i}$ with $a_i\in L$, we define its \defi{adjoint} by $f^*=\sum_{i\in \ZZ} a_i^{1/q^i} \tau^{-i}$.   For $f_1,f_2\in L\twistedseries$, we have $(f_1f_2)^*=f_2^* f_1^*$ and $(f_1^*)^*=f_1$.  Define the map 
\[
\varphi\colon \OO_\infty \to L[\![\tau]\!],\quad x \mapsto \phi_x^*.
\]   
Using that $\OO_\infty$ is commutative, we find that $\varphi$ is a homomorphism that satisfies
\[
\ord_{\tau} \varphi(x)= n d_\infty \ord_\infty(x)
\]
for all $x\in \OO_\infty$.    In the language of \cite{MR0384707}*{\S1D}, $\varphi$ is a \defi{formal $\OO_\infty$-module} of \defi{height} $n$.   

If one fixes a formal $\OO_\infty$-module $\iota\colon \OO_\infty \to \kbar[\![\tau]\!]$, then by \cite{MR0384707}*{Prop.~1.7(1)} there is a $v\in \Lbar[\![\tau]\!]^\times$ such that $v^{-1}\varphi(x)v = \iota(x)$ for $x\in \OO_\infty$.   Let $\mathcal{D}_\varphi$ be the centralizer of $\varphi(\OO_\infty)$ in $\Lbar(\!(\tau)\!)$.  By \cite{MR0384707}*{Prop.~1.7(2)}, $\mathcal{D}_\varphi$ is a central $F_\infty$-division algebra with invariant $1/n$ and $\mathcal{D}_\varphi\cap \Lbar[\![\tau]\!]$ is the ring of integers of $\mathcal{D}_\varphi$.   One can then define a continuous homomorphism 
\[
\varrho \colon W_L \to \mathcal{D}_\varphi^\times,\quad \sigma\mapsto \sigma(v)\tau^{\deg(\sigma)} v^{-1}.
\]
For $\sigma\in W_L$, we have $\rho_\infty(\sigma) = (\varrho(\sigma)^*)^{-1}.$      Note that this construction works for any formal $\OO_\infty$-module $\OO_\infty\to L[\![\tau]\!]$ with height $1\leq n<\infty$.

\section{Drinfeld modules of rank 1} \label{SS:rank 1}
Let $\phi\colon A \to L[\tau]$ be a Drinfeld module of rank $1$ with generic characteristic.   For a place $\lambda \neq \infty$ of $F$, the Tate module $T_\lambda(\phi)$ is a free $\OO_\lambda$-module of rank 1.   The Galois action on $T_\lambda(\phi)$ commutes with the $\OO_\lambda$-action, and hence our Galois representation $\rho_\lambda$ is of the form
\[
\rho_\lambda \colon \Gal_L \to \Aut_{\OO_\lambda}(T_\lambda(\phi)) = \OO_\lambda^\times.
\]
For the place $\infty$, we have defined a representation
\[
\rho_\infty \colon W_L \to D_\phi^\times =  F_\infty^\times,
\]
where $D_\phi$ equals $F_\infty$ since it is a central $F_\infty$-division algebra with invariant $-1$.  In this section, we will prove the following proposition, whose corollary is the rank 1 case of Theorem~\ref{T:main theorem}.   

\begin{proposition} \label{P:complex multiplication case}
Let $\phi \colon A \to L[\tau]$ be a Drinfeld module of rank 1 with generic characteristic and assume that $L$ is a finitely generated field. 
Then the group $\big(\prod_{\lambda} \rho_\lambda\big)(W_L)$ is an open subgroup with finite index in  $\big(\prod_{\lambda\neq \infty}\OO_\lambda^\times\big) \times F_\infty^\times$.
\end{proposition}  

\begin{corollary} \label{C:main theorem n=1}
Let $\phi \colon A \to L[\tau]$ be a Drinfeld module of rank 1 with generic characteristic and assume that $L$ is a finitely generated field.  Then $\rho_\infty(W_L)$ is an open subgroup with finite index in $D_\phi^\times = F_\infty^\times$.
\end{corollary}

\subsection{Proof of Proposition~\ref{P:complex multiplication case}}   Since $\phi$ has generic characteristic, it induces an embedding $F\to L$ that we view as an inclusion.    The following lemma allows us to reduce to the case where $L$ is a global function field and $L/F$ is an abelian extension.

\begin{lemma}  
If Proposition~\ref{P:complex multiplication case} holds in the special case where $L$ is a finite separable abelian extension of $F$, then the full proposition holds.
\end{lemma}
\begin{proof}
Let $H_A$ be the maximal unramified abelian extension of $F$ in $F^\sep$ for which the place $\infty$ splits completely; it is a finite abelian extension of $F$.    Choose an embedding $H_A\subseteq L^\sep$.    By \cite{MR1196509}*{\S15} (and our generic characteristic and rank 1 assumptions on $\phi$), there is a Drinfeld module $\phi'\colon A \to H_A[\tau]$ such that $\phi$ and $\phi'$ are isomorphic over $L^\sep$ (since $L$ is a finitely generated extension of $F$, we can choose an embedding of $L$ into the field $\mathbf{C}$ of loc.~cit.).    Moreover, there is a finite extension $L'$ of $LH_A$ such that $\phi$ and $\phi'$ are isomorphic over $L'$.    Therefore, $(\prod_{\lambda} \rho_{\phi,\lambda})(W_{L'})$ and $(\prod_{\lambda} \rho_{\phi',\lambda})(W_{L'})$ are equal in $(\prod_{\lambda\neq\infty} \OO_\lambda^\times) \times F_\infty^\times$.

By the hypothesis of the lemma, we may assume that $(\prod_{\lambda} \rho_{\phi',\lambda})(W_{H_A})$ is an open subgroup of finite index in $(\prod_{\lambda\neq\infty} \OO_\lambda^\times) \times F_\infty^\times$.   Replacing $H_A$ by the finitely generated extension $L'$, we find that $(\prod_{\lambda} \rho_{\phi',\lambda})(W_{L'})$ is still an open subgroup of finite index in $(\prod_{\lambda\neq\infty} \OO_\lambda^\times) \times F_\infty^\times$ (though possibly of larger index).   Therefore, $(\prod_{\lambda} \rho_{\phi,\lambda})(W_{L})$ contains $(\prod_{\lambda} \rho_{\phi,\lambda})(W_{L'})=(\prod_{\lambda} \rho_{\phi',\lambda})(W_{L'})$ which is open and of finite index in $(\prod_{\lambda\neq\infty} \OO_\lambda^\times) \times F_\infty^\times$.
\end{proof}

By the above lemma, we may assume without loss of generality that $L$ is a finite separable and abelian extension of $F$.   The benefit of $L$ being a global function field is that we will be able to use class field theory.  Since $\rho_\lambda|_{W_L}$ is continuous with commutative image, it factors through the maximal abelian quotient $W_L^\ab$ of $W_L$.  Let $\bfA_L^\times$ be the group of ideles of $L$.   For each place $\lambda$ of $F$, we define the continuous homomorphism
\[
\widetilde{\rho}_\lambda \colon\bfA_L^\times \to W_L^\ab \xrightarrow{\rho_\lambda} F_\lambda^\times
\]
where the first homomorphism is the Artin map of class field theory.   The homomorphism $\widetilde{\rho}_\lambda$ is trivial on $L^\times$, and has image in $\OO_\lambda^\times$ when $\lambda\neq \infty$.   Define $L_\lambda:= L\otimes_F F_\lambda$ and let $N_\lambda \colon L_\lambda \to F_\lambda$ be the corresponding norm map.  Define the continuous homomorphism 
\[
\chi_\lambda \colon  \bfA_L^\times \to F_\lambda^\times,\quad \alpha\mapsto\widetilde{\rho}_\lambda(\alpha) N_\lambda(\alpha_\lambda)
\] 
where $\alpha_\lambda$ is the component of $\alpha$ in $L_\lambda^\times = \prod_{v | \lambda} L_v^\times$.

Let $S$ be the set of places of $L$ for which $\phi$ has bad reduction or which lie over $\infty$.   For $v\notin S$, let $\lambda_v$ be the place of $F$ lying under $v$.    For each place $v\notin S$ of $L$, define $\pi_v := \rho_\infty(\Frob_v)$.  By Lemma~\ref{L:infty independence}, $\pi_v$ belongs to $F^\times$; it also equals $\rho_\lambda(\Frob_v)$ for all $\lambda\neq \lambda_v$.    For each place $v\notin S$ of $L$ and $\lambda$ of $F$, we have 
\begin{equation} \label{E:HT}
\ord_{\lambda}(\pi_v) = \begin{cases}
	[\FF_v:\FF_{\lambda_v}]  & \text{ if $\lambda=\lambda_v$}, \\
	-[\FF_v:\FF_{\infty}]  & \text{ if $\lambda=\infty$}, \\
      0 & \text{otherwise},
\end{cases}
\end{equation} 
cf.~\cite{MR0439758}*{Proposition~2.1}.
We now show that $\chi_\lambda$ is independent of $\lambda$.

\begin{lemma} \label{L:Hecke character}
There is a unique character $\chi \colon  \bfA_L^\times \to F^\times$ that satisfies the following conditions:
\begin{alphenum}
\item $\ker(\chi)$ is an open subgroup of $ \bfA_L^\times$.
\item If $\alpha\in L^\times$, then $\chi(\alpha) = N_{L/F}(\alpha)$.
\item If $\alpha=(\alpha_v)$ is an idele with $\alpha_v =1$ for $v\in S$, then $\chi(\alpha) = \prod_{v \not\in S} \pi_v^{\ord_v(\alpha_v)}$.
\end{alphenum}
For every place $\lambda$ of $F$, we have $\chi_\lambda(\alpha)=\chi(\alpha)$ for all $\alpha\in \bfA_L^\times$.  
\end{lemma}
\begin{proof}
Fix a place $\lambda$ of $F$.   If $\alpha\in L^\times$, then $\chi_\lambda(\alpha)=N_\lambda(\alpha)= N_{L/F}(\alpha)$ since $\widetilde{\rho}_\lambda$ is trivial on $L^\times$.   Let $S_\lambda$ be those places of $L$ that belong to $S$ or lie over $\lambda$.  For an idele $\alpha\in \bfA_L^\times$ satisfying $\alpha_v=1$ for $v\in S_\lambda$, we have $\chi_\lambda(\alpha) = \widetilde{\rho}_\lambda(\alpha)$ which equals $ \prod_{v \not\in S_\lambda} \rho_{\lambda}(\Frob_v)^{\ord_v(\alpha_v)}$ since $\rho_\lambda$ is unramified outside $S_\lambda$.   Therefore, $\chi_\lambda(\alpha)=\prod_{v \not\in S_\lambda} {\pi_v}^{\ord_v(\alpha_v)}$.

Define $U=\prod_v \OO_v^\times$; it is an open subgroup of $\bfA^\times_L$.  Consider an idele $\beta \in U$ that satisfies $\beta_v= b \in L^\times$ for all $v\in S_\lambda$.   We then have
\[
\chi_\lambda(\beta) = \chi_\lambda(b)\chi_{\lambda}(b^{-1}\beta) = N_{L/F}(b)\prod_{v\notin S_\lambda} \pi_v^{\ord_v(b^{-1}\beta_v)}=N_{L/F}(b)\prod_{v\notin S_\lambda} \pi_v^{-\ord_v(b)},
\]
which is an element of $F^\times$.  Take a place $\lambda'\neq \infty$ of $F$.  By (\ref{E:HT}) and using that $\ord_v(b)=0$ for $v\in S_\lambda$, we have
\begin{align*}
\ord_{\lambda'}\Big(\prod_{v\notin S_\lambda} \pi_v^{\ord_v(b)}\Big) & = \sum_{v|\lambda'}\ord_v(b) \ord_{\lambda'}(\pi_v)\\
&=\sum_{v|\lambda'}[\FF_v:\FF_{\lambda'}]\ord_v(b)\\
&=\sum_{v|\lambda'}\ord_{\lambda'} N_{L_v/F_{\lambda'}}(b) = \ord_{\lambda'} N_{L/F}(b).
\end{align*}
Therefore, $\ord_{\lambda'}(\chi_\lambda(\beta))=0$ for all $\lambda'\neq \infty$, and hence $\chi_\lambda(\beta)$ belongs to $A^\times = k^\times$.   By weak approximation, the ideles $\beta\in U$ with $\beta_v=b \in L^\times$ for all $v\in S_\lambda$ are dense in $U$.  Since $\chi_\lambda$ is continuous, we deduce that $\chi_\lambda(U)\subseteq k^\times$ and hence $\ker(\chi_\lambda)$ is an open subgroup of $\bfA_L^\times$.     The group $\bfA^\times_L/\ker(\chi_\lambda)$ is generated by $L^\times$ and ideles with 1 at the places $v\in S_\lambda$, so $\chi_\lambda$ takes values in $F^\times$.

Define $\chi:=\chi_\infty$.  We have just seen that $\chi$ takes values in $F^\times$ and satisfies conditions (a), (b) and (c).  Now suppose that $\chi'\colon \bfA^\times_L \to F^\times$ is a group homomorphism that satisfies the following conditions:
\begin{itemize}
\item $\ker(\chi')$ is an open subgroup of $ \bfA_L^\times$.
\item If $\alpha\in L^\times$, then $\chi'(\alpha) = N_{L/F}(\alpha)$.
\item There is a finite set $S'\supseteq S$ of places of $L$ such that $\chi'(\alpha) = \prod_{v \not\in S'} {\pi_v}^{\ord_v(\alpha_v)}$ for all ideles $\alpha$ with $\alpha_v=1$ for $v\in S'$.   
\end{itemize}
The character $\chi'$ is determined by its values on the group $\bfA_L^\times/(\ker(\chi')\cap\ker(\chi))$, and this group is generated by $L^\times$ and the ideles with $v$-components equal to $1$ for $v\in S'$.  Since $\chi$ and $\chi'$ agree on such elements, we find that $\chi'=\chi$.   This proves the uniqueness of a character satisfying conditions (a), (b) and (c).   With $\chi'=\chi_\lambda$ and $S'=S_\lambda$, we conclude that $\chi_\lambda=\chi$.
\end{proof}

Let $C_F$ and $C_L$ be the idele class groups of $F$ and $L$, respectively.  The natural quotient map $\big(\prod_{\lambda\neq \infty}\OO_\lambda^\times\big) \times F_\infty^\times \to C_F$ has kernel $k^\times$ and its image is an open subgroup of finite index in $C_F$.   Since the $\widetilde{\rho}_\lambda$ are trivial on $L^\times$, we can define a homomorphism $f\colon C_L \to C_F$ that takes the idele class containing $\alpha\in \bfA_L$ to the idele class of $C_F$ containing $(\widetilde{\rho}_\lambda(\alpha))_\lambda$.     To prove the proposition, it suffices to show that the image of $f$ is open with finite index in $C_F$.   By the definition of the $\chi_\lambda$ and Lemma~\ref{L:Hecke character}, we have
\[
(\widetilde{\rho}_\lambda(\alpha))_\lambda= (\chi_\lambda (\alpha)N_\lambda(\alpha_\lambda)^{-1})_\lambda=\chi(\alpha)( N_\lambda(\alpha_\lambda)^{-1})_\lambda
\]
Therefore, $f(\alpha)=N_{L/F}(\alpha)^{-1}$ for all $\alpha \in C_L$, where $N_{L/F}\colon C_L \to C_F$ is the norm map.   Class field theory tells us that $N_{L/F}(C_L)$ is an open subgroup of $C_F$ and the index $[C_F:N_{L/F}(C_L)]$ equals $[L:F]$; the same thus holds for $f$.

\begin{remark} \label{R:CFT Hayes}
Consider the special case where $A=k[t]$, $F=k(t)$, and $\phi\colon k[t]\to F[\tau]$ is the Carlitz module of Example~\ref{Ex:CFT Hayes}.  As noted in Example~\ref{Ex:CFT Hayes}, we have a continuous homomorphism
\[
\beta\colon W_F^{\ab} \to \Big(\prod_{\lambda\neq \infty} \OO_\lambda^\times\Big) \times \ang{t}\cdot (1+t^{-1}\OO_\infty)\xrightarrow{\sim} C_F
\]
where the first map is $\prod_{\lambda} \rho_\lambda$ and the second is the quotient map.   Composing $\beta$ with the Artin map of $F$, we obtain a homomorphism $f\colon C_F\to C_F$ which from the calculation above is $f(\alpha)=N_{F/F}(\alpha)^{-1}=\alpha^{-1}$.    Therefore, $W_F^\ab \to C_F,\, \sigma\mapsto \beta(\sigma^{-1})$ is the inverse of the Artin map for $F$ as claimed in Example~\ref{Ex:CFT Hayes}.
\end{remark}

\section{Tate conjecture}  \label{S:Tate}

Let $\phi\colon A \to L[\tau]$ be a Drinfeld module of rank $n$ and let $D_\phi$ be the centralizer of $\phi(A)$ in $\Lbar\twistedLaurent$.  Using the extended map $\phi\colon F_\infty \to L^\perf\twistedLaurent$, we have shown that $D_\phi$ is a central $F_\infty$-division algebra with invariant $-1/n$.   In \S\ref{S:construction}, we  constructed a continuous representation 
\[
\rho_\infty\colon W_L \to D_\phi^\times.
\]
We can view $\End_L(\phi)\otimes_A F_\infty$ as a $F_\infty$-subalgebra of $D_\phi$; it commutes with the image of $\rho_\infty$.  The following $\infty$-adic analogue of the Tate conjecture, says that $\End_L(\phi)\otimes_A F_\infty$ is precisely the centralizer of $\rho_\infty(W_L)$ in $D_\phi$, at least assuming that $L$ is finitely generated and $\phi$ has generic characteristic.

\begin{theorem} \label{T:Tate conjecture}  Let $\phi\colon A \to L[\tau]$ be a Drinfeld module with generic characteristic and $L$ a {finitely generated} field.   Then the centralizer of $\rho_\infty(W_L)$ in $D_\phi$ is $\End_L(\phi)\otimes_A F_\infty$.
\end{theorem}

For the rest of the section, assume that $L$ is a {finitely generated} field.   Recall that for a place $\lambda\neq \infty$, the $\lambda$-adic version of the Tate conjecture says that the natural map
\begin{equation} \label{E:lambda-adic Tate}
\End_L(\phi) \otimes_A F_\lambda \to \End_{F_\lambda[\Gal_L]}(V_\lambda(\phi))
\end{equation}
is an isomorphism.  This is a special case of theorems proved independently by Taguchi \cite{MR1286009} and Tamagawa \cite{MR1417504}; we will make use of Tamagawa's more general formulation.  We can give $\End_{F_\lambda}(V_\lambda(\phi))$ a $\Gal_L$-action by $\sigma(f):= \rho_\lambda(\sigma) \circ f \circ \rho_{\lambda}(\sigma)^{-1}$.   That (\ref{E:lambda-adic Tate}) is an isomorphism is equivalent to having $\End_{F_\lambda}(V_\lambda(\phi))^{\Gal_L}=\End_L(\phi)\otimes_A F_\lambda$.   For the $\infty$-adic version, the ring $D_\phi$ has a natural $\Gal(\Lbar/L)$-action and the subring $D_\phi^{\Gal(\Lbar/L)} = \Cent_{L^\perf\twistedLaurent}(\phi(A))$ certainly contains $\End_L(\phi)\otimes_A F_\infty$; we will show that they are equal, and from the following lemma, deduce Theorem~\ref{T:Tate conjecture}.

\begin{lemma} \label{L:reduction of Tate to invariants}
If $D_\phi^{\Gal(\Lbar/L)}=\End_{L}(\phi)\otimes_A F_\infty$, 
then the centralizer of $\rho_\infty(W_L)$ in $D_\phi$ is $\End_L(\phi)\otimes_A F_\infty$.
\end{lemma}
\begin{proof}
Fix an $f\in D_\phi$ that commutes with $\rho_\infty(W_L)$.   Take any $\sigma\in W_L$.  The series $f$ and $\rho_\infty(\sigma)$ commute, so we have
\[
\sigma(u)\tau^{\deg(\sigma)} u^{-1} \cdot f = f \cdot \sigma(u)\tau^{\deg(\sigma)} u^{-1},
\]
where $u$ is a series as in Lemma~\ref{L:u facts}(\ref{I:u fact i}).  Therefore,
\[
\sigma(u)^{-1} f \sigma(u) = \tau^{\deg(\sigma)} (u^{-1} f u) \tau^{-\deg(\sigma)}= \sigma(u^{-1} f u)
\] 
where the last equality uses that $u^{-1}fu$ has coefficients in $\kbar$.    Since $W_L$ is dense in $\Gal_L$, we have $\sigma(u)^{-1} f \sigma(u) =\sigma(u^{-1} f u)$ for all $\sigma\in \Gal_L$ and hence also for all $\sigma\in \Gal(\Lbar/L)$.    Therefore, $\sigma(u)^{-1} f \sigma(u) = \sigma(u)^{-1} \sigma(f) \sigma(u)$ and hence $\sigma(f)=f$, for all $\sigma\in \Gal(\Lbar/L)$.  So $f$ belongs to $D_\phi^{\Gal(\Lbar/L)}$ and is thus an element of $\End_L(\phi)\otimes_A F_\infty$ by assumption.  This proves that the centralizer of $\rho_\infty(W_L)$ in $D_\phi$ is contained in $\End_{L}(\phi)\otimes_A F_\infty$.; we have already noted that the other inclusion holds.
\end{proof}

The rest of \S\ref{S:Tate} is dedicated to proving Theorem~\ref{T:Tate conjecture}.   To relate our construction with the work of Tamagawa, it will be useful to replace $\phi$ with its corresponding $A$-motive.   We give enough background to prove the theorem; this material will not be needed outside \S\ref{S:Tate}.

\subsection{\'Etale $\tau$-modules}
Let $L$ be an extension field of $k$ (as usual, $k$ is a fixed finite field with $q$ elements).  Let $L\laurent$ be the (commutative) ring of Laurent series in $t^{-1}$ with coefficients in $L$.  Define the ring homomorphism
\[
\sigma \colon L\laurent \to L\laurent,\quad \sum_{i} c_i t^{-i} \mapsto \sum_i c_i^q t^{-i}.
\]   
Let $R$ be a subring of $L\laurent$ that is stable under $\sigma$; for example, $L[t]$, $L(t)$ and $L\laurent$. 

 \begin{definition}
 A \defi{$\tau$-module} over $R$ is a pair $(M,\tau_M)$ consisting of an $R$-module $M$ and a $\sigma$-semilinear map $\tau_M\colon M \to M$ (i.e., $\tau_M$ is additive and satisfies $\tau_M(fm)=\sigma(f)\tau_M(m)$ for all $f\in R$ and $m\in M$).  A morphism of $\tau$-modules is an $R$-module homomorphism that is compatible with the $\tau$ maps. 
\end{definition}

When convenient, we shall denote a $\tau$-module $(M,\tau_M)$ simply by $M$.   We can view $R$ as a $\tau$-module over itself by setting $\tau_R=\sigma|_R$.    For an $R$-module $M$, denote by $\sigma^*(M)$ the scalar extension $R\otimes_{\sigma,R} M$ of $M$ by $\sigma\colon R\to R$. Giving a $\sigma$-semilinear map $\tau_M\colon M\to M$ is thus equivalent to giving an $R$-linear map 
\[
\tau_{M,\lin}\colon \sigma^*(M)\to M
\] 
which we call the \defi{linearization} of $\tau_M$.   We say that a $\tau$-module $M$ over $R$ is \defi{\'etale} if $M$ is a free $R$-module of finite rank and the linearization $\tau_{M,\lin}\colon \sigma^*(M)\to M$ is an isomorphism.  

Let $M_1$ and $M_2$ be $\tau$-modules over $R$.  We define $M_1\otimes_R M_2$ to be the $\tau$-module whose underlying $R$-module is $M_1\otimes_R M_2$ with $\tau$ map determined by $\tau_{M_1 \otimes_R M_2}(m_1\otimes m_2) = \tau_{M_1}(m_1) \otimes \tau_{M_2}(m_2)$.  Now suppose that $M_1$ is \'etale.  Define the $R$-module $H:=\Hom_{R}(M_1,M_2)$.   Let $\tau_H\colon H \to H$ be the $\sigma$-semilinear that corresponds to the $R$-linear map
 \[
 \sigma^*(H)\to H,\quad f\mapsto \tau_{M_2,\lin} \circ f \circ \tau_{M_1,\lin}^{-1},
 \]
 where we are using the natural isomorphism $\sigma^*(H)\cong \Hom_{R}(\sigma^*(M_1), \sigma^*(M_2))$.     The pair $(H,\tau_H)$ is a $\tau$-module over $R$.  If  $M_1$ and $M_2$ are both \'etale over $R$, then $H$ is also \'etale.
 
Suppose that $R\subseteq R'$ are subrings of $L\laurent$ which are stable under $\sigma$.  Let $M$ be a $\tau$-module over $R$.   We can then give $R'\otimes_R M$ the structure of a $\tau$-module over $R'$.  If $M$ is \'etale, then $R'\otimes_R M$ is an \'etale $\tau$-module over $R'$.
 
For a $\tau$-module $M$, let $M^{\tau}$ be the group of $m\in M$ for which $\tau_M(m)=m$; it is a module over the ring $R_0:=\{r\in R: \sigma(r)=r\}$ (for $R=L(t)$ and $L\laurent$, we have $R_0=k(t)$ and $k\laurent$, respectively).      Let $H$ be the $\tau$-module $\Hom_R(M_1,M_2)$ where $M_1$ and $M_2$ are $\tau$-modules and $M_1$ is \'etale;  then $H^\tau$ agrees with the set $\Hom(M_1,M_2)$ of endomorphisms $M_1\to M_2$ of $\tau$-modules.

\subsection{Weights}
Fix a separably closed extension $K$ of $k$.  We shall describe the \'etale $\tau$-modules over $K\laurent$; it turns out that the category of such $\tau$-modules is semisimple, we first define the simple ones.

\begin{definition} \label{D:N-lambda}
Let $\lambda=s/r$ be a rational number with $r$ and $s$ relatively prime integers and $r\geq 1$.  Define the free $K\laurent$-module
\[
N_\lambda := K\laurent e_1 \oplus \cdots \oplus K\laurent e_r.
\]
Let $\tau_\lambda \colon N_\lambda \to N_\lambda$ be the $\sigma$-semilinear map that satisfies $\tau_\lambda(e_i)=e_{i+1}$ for $1\leq i<r$ and $\tau_\lambda(e_r)=t^s e_1$.  The pair $(N_\lambda, \tau_\lambda)$ is an \'etale $\tau$-module over $K\laurent$.
\end{definition}

\begin{proposition}  \label{P:Dieudonne classification}
\begin{romanenum}
\item
If $M$ is an \'etale $\tau$-module over $K\laurent$, then there are unique rational numbers $\lambda_1\leq \lambda_2 \leq \cdots\leq \lambda_m$ such that 
\begin{itemize}
\item $M\cong N_{\lambda_1}\oplus \cdots \oplus N_{\lambda_m}$.   
\item the $t^{-1}$-adic valuations of the roots of the characteristic polynomial of $\tau_M$ expressed on any $K\laurent$-basis of $M$ are $\{-\lambda_i\}_i$, with each $\lambda_i$ counted with multiplicity $\dim N_\lambda$.
\end{itemize}

\item
For $\lambda\in\QQ$, the ring $\End(N_\lambda)$ is a central $k\laurent$-division algebra with Brauer invariant $\lambda$.
\end{romanenum}
\end{proposition}
\begin{proof}
This follows from \cite{Laumon}*{Appendix~B}; although the proposition is only proved for a particular field $K$, nowhere do the proofs make use of anything stronger then $K$ being separably closed.    This was observed by Taelman in \cite{MR2468475}*{\S5}; his notion of a ``Dieudonn\'e $t$-module'' corresponds with \'etale  $\tau$-modules over $K\laurent$ (in Definition~5.1.1 of loc.~cit.~one should have $K\laurent \sigma(V)=V$).
\end{proof}

We call the rational numbers $\lambda_i$ of Proposition~\ref{P:Dieudonne classification}(i) the \defi{weights} of $M$.   If all the weights of $M$ equal $\lambda$, then we say that $M$ is \defi{pure of weight $\lambda$}.

\begin{lemma} \label{L:lattice criterion}  \cite{MR2468475}*{Prop.~5.14}
Fix a rational number $\lambda=s/r$ with $r$ and $s$ relatively prime integers and $r\geq 1$.  Let $M$ be an \'etale $\tau$-module over $K\laurent$ with $K$ an algebraically closed extension of $k$.  The following are equivalent:
\begin{itemize}
\item 
$M$ is pure of weight $\lambda$.
\item there exists a $K[\![t^{-1}]\!]$-lattice $\Lambda\subseteq M$ such that $\tau_M^r(\Lambda)=t^s\Lambda$.  
\end{itemize}
\end{lemma}

Let $L$ be a field extension of $k$ (not necessarily separably closed) and let $M$ be an \'etale $\tau$-module over $L(t)$.   The \defi{weights} of $M$ are the weights of the $\tau$-module $K\laurent \otimes_{L(t)} M$ over $K\laurent$ where $K$ is any separably closed field containing $L$.   Again, we say that $M$ is \defi{pure of weight $\lambda$} if all the weights of $M$ equal $\lambda$.   We now give a criterion for $M$ to be pure of weight 0.

\begin{lemma} \label{L:pure 0 criterion}
Define the subring $\OO:=L[t^{-1}]_{(t^{-1})}=L(t)\cap L[\![t^{-1}]\!]$ of $L(t)$; it is a local ring with quotient field $L(t)$.  Let $M$ be an \'etale $\tau$-module over $L(t)$.   Then the following are equivalent:
\begin{alphenum}
\item $M$ is pure of weight 0.
\item There is an $\OO$-submodule $N$ of $M$ stable under $\tau_M$ such that $(N,\tau_M|_N)$ is an \'etale $\tau$-module over $\OO$ and the natural map $L(t)\otimes_{\OO} N \to M$ of $\tau$-modules is an isomorphism.   
\end{alphenum}
\end{lemma}
\begin{proof}
First suppose that $M$ is pure of weight 0.   By Lemma~\ref{L:lattice criterion}, there is an $\Lbar[\![t^{-1}]\!]$-lattice $\Lambda$ of $\Lbar\laurent\otimes_{L(t)} M$ such that $\tau_M(\Lambda)=\Lambda$.   Fix a basis $e_1,\ldots, e_d$ of $M$ over $L(t)$; we may assume that the $e_i$ are contained in $\Lambda$.  Let $N$ be the $\OO$-submodule of $M$ generated by the set $\calB=\{\tau_M^j(e_i): 1\leq i \leq d,\, j\geq 1\}$.   We can write each $v\in \calB$, uniquely in the form $v=\sum_{i} a_i e_i$ with $a_i\in L(t)$; let $\alpha$ be the infimum of the values $\ord_{t^{-1}}(a_i)$   over all $i\in \{1\ldots, d\}$ and $v\in \calB$.       We find that $\alpha$ is finite, since $N$ is contained in the $\Lbar[\![t^{-1}]\!]$-lattice $\Lambda$ which is stable under $\tau_M$.  Using that $\alpha$ is finite, we find that $N$ is a free $\OO$-module of rank $d$ which is stable under $\tau_M$ and that the map $L(t)\otimes_{\OO} N \to M$ is an isomorphism.  The $\tau$-module $(N,\tau_M|N)$  is \'etale since $(M,\tau_M)$ is \'etale.   It is now clear that $N$ satisfies all the conditions of (b).

Now suppose there is an $\OO$-submodule $N$ satisfying the conditions of (b).   Then $\Lambda:= \Lbar[\![t^{-1}]\!] \otimes_{\OO} N$ is a $\Lbar[\![t^{-1}]\!]$-lattice in $\Lbar\laurent\otimes_{L(t)} M$ that satisfies $\tau_{\Lbar\laurent\otimes_{L(t)} M}(\Lambda)=\Lambda$.  Lemma~\ref{L:lattice criterion} implies that $\Lbar\laurent\otimes_{L(t)} M$, and hence $M$ also, is pure of weight 0.
\end{proof}

\subsection{Tate conjecture}

Let $M$ be an \'etale $\tau$-module over $L(t)$.   The group $\Gal_L$ acts on $M':=L^\sep\laurent \otimes_{L(t)} M$ via its action on the coefficients of $L^\sep\laurent$.   The $\Gal_L$-action on $M'$ commutes with $\tau_{M'}$, so $M'^{\tau}$ is a vector space over $k\laurent$ with an action of $\Gal_L$.

\begin{theorem} \label{T:Tamagawa-Tate}
Let $L$ be a finitely generated extension of $k$.   Let $M$ be an \'etale $\tau$-module over $L(t)$ that is pure of weight 0.  Then the natural map
\[
M^{\tau} \otimes_{k(t)} k\laurent \to \big((L^\sep\laurent \otimes_{L(t)}  M)^\tau \big)^{\Gal_L}
\]
is an isomorphism of finite dimensional $k\laurent$-vector spaces.
\end{theorem}
\begin{proof}
For an \'etale $\tau$-module $M$ over $L(t)$, we define
\[
\widehat{V}(M):= (L^\sep\laurent \otimes_{L(t)} M)^\tau;
\]
it is a $k\laurent$-vector space with a natural action of $\Gal_L$.  Let $M'$ and $M$ be \'etale $\tau$-modules over $L(t)$ that are pure of weight 0.    There is a natural homomorphism
\begin{equation} \label{E:Tamagama-Tate}
\Hom(M', M)\otimes_{k(t)} k\laurent \to \Hom_{k\laurent[\Gal_L]}\big(\widehat{V}(M'), \widehat{V}(M)\big)
\end{equation}
of vector spaces over $k\laurent$.    We claim that (\ref{E:Tamagama-Tate}) is an isomorphism.   In the notation of Tamagawa in \cite{MR1417504}, $M'$ and $M$ are  ``restricted $L(t)\{\tau\}$-modules that are \'etale at $t^{-1}=0$''.   That $M$ is an \'etale $\tau$-module over $L(t)$ is equivalent to it being a ``restricted module over $L(t)\{\tau\}$'', and it further being pure of weight 0 is equivalent to it being ``\'etale at $t^{-1}=0$''  by Lemma~\ref{L:pure 0 criterion}.   (Note that in Definition 1.1 of \cite{MR1417504}, the submodule $\mathcal{M}$ should also be an $O_{L(t)}$-sublattice of $M$).     Tamagawa's analogue of the Tate conjecture \cite{MR1417504}*{Theorem~2.1} then says that (\ref{E:Tamagama-Tate}) is an isomorphism of (finite dimensional) vector spaces over $k\laurent$.  Tamagawa theorem, whose proof is based on methods arising from $p$-adic Hodge theory, are only sketched in \cite{MR1417504}; details have since been provided by N.~Stalder \cite{MR2644927}.

Now consider the special case where the $\tau$-module $M'$ is $L(t)$ with $\tau_{M'}=\sigma|_{L(t)}$.     So $\widehat{V}(M')= k\laurent$ with the trivial $\Gal_L$-action.  We have isomorphisms
\[
\Hom(M',M)=\Hom_{L(t)}(L(t), M)^\tau \xrightarrow{\sim} M^\tau,\quad f \mapsto f(1)
\]
and 
\[
\Hom_{k\laurent[\Gal_L]}\big(\widehat{V}(M'), \widehat{V}(M)\big)=\Hom_{k\laurent[\Gal_L]}\big(k\laurent, \widehat{V}(M)\big)\xrightarrow{\sim} \widehat{V}(M)^{\Gal_L},\quad f\mapsto f(1).
\]
Combining with the isomorphism (\ref{E:Tamagama-Tate}), we find that the natural map 
\[
M^\tau\otimes_{k(t)} k\laurent \to \widehat{V}(M)^{\Gal_L} = \big((L^\sep\laurent \otimes_{L(t)}  M)^\tau \big)^{\Gal_L}
\] 
is an isomorphism of finite dimensional vector space over $k\laurent$.
\end{proof}

\begin{corollary} \label{C:Tamagawa-Tate}
Let $L$ be a finitely generated extension of $k$.   Let $M$ be an \'etale $\tau$-module over $L(t)$ that is pure of some weight $\lambda$.  Then for any separably closed extension $K$ of $L$, the natural map
\[
\End(M)\otimes_{k(t)} k\laurent \to \End(K\laurent\otimes_{L(t)}M)^{\Gal(K/L)}
\]
is an isomorphism.
\end{corollary}
\begin{proof}
Fix an embedding $L^\sep\subseteq K$.   We have an inclusion $ \End(L^\sep\laurent\otimes_{L(t)}M) \subseteq  \End(K\laurent\otimes_{L(t)}M)$ of finite dimensional vector spaces over $k\laurent$; it is actually an equality since by Proposition~\ref{P:Dieudonne classification}(ii), their dimensions depend only the weights of $M$.  Hence,
\[
 \End(K\laurent\otimes_{L(t)}M)^{\Gal(K/L)}= \End(L^\sep\laurent\otimes_{L(t)}M)^{\Gal(K/L)}= \End(L^\sep\laurent\otimes_{L(t)}M)^{\Gal_L}.
 \]
 
So without loss of generality, we may assume that $K=L^\sep$.  Define the $L(t)$-module $H=\End_{L(t)}(M)$.  Since $M$ is an \'etale $\tau$-module, we can give $H$ the structure of \'etale $\tau$-module over $L(t)$.   The natural map $H^\tau \otimes_{k(t)} k\laurent \to \big((L^\sep\laurent \otimes_{L(t)}  H)^\tau \big)^{\Gal_L}$ can be rewritten as
\[
\End(M)\otimes_{k(t)} k\laurent \to \End(L^\sep\laurent\otimes_{L(t)}M)^{\Gal_L}.
\]
So the corollary will follow from Theorem~\ref{T:Tamagawa-Tate} if we can show that $H$ is pure of weight 0.  

The dual $M^{\vee}:=\Hom_{L(t)}(M,L(t))$ is an \'etale $\tau$-module over $L(t)$ that is pure of weight $-\lambda$ (for the weight, one can use the characterization in terms of eigenvalues as in Proposition~\ref{P:Dieudonne classification}).   If $M_1$ and $M_2$ are \'etale $\tau$-modules over $L(t)$ that are pure of weight $\lambda_1$ and $\lambda_2$, respectively, then $M_1\otimes_{L(t)} M_2$ is pure of weight $\lambda_1+\lambda_2$ (use Lemma~\ref{L:lattice criterion}).      Therefore $H$, which is isomorphic as a $\tau$-module to $M^{\vee}\otimes_{L(t)} M$, is pure of weight $-\lambda + \lambda=0$.
\end{proof}

\subsection{Proof of Theorem~\ref{T:Tate conjecture}}
Let $\phi\colon A \to L[t]$ be a Drinfeld module with generic characteristic and $L$ a finitely generated field.   \\

\noindent{\textbf{Case 1:}} \emph{Suppose that $A=k[t]$ and $F=k(t)$.}

Define $M_\phi:= L[\tau]$ and give it the $L[t]=L\otimes_k A$-module structure for which
\[
(c\otimes a) \cdot m = cm\phi_a
\]
for $c\in L$, $a\in A$ and $m\in M_\phi$.    Define the map $\tau_{M_\phi}\colon M_\phi \to M_\phi$ by $m\mapsto \tau m$.      The pair $(M_\phi,\tau_{M_\phi})$ is a $\tau$-module over $L[t]$.   As an $L[t]$-module, $M_\phi$ is free of rank $n$ with basis $\beta=\{1,\tau,\ldots, \tau^{n-1}\}$.  With respect to the basis $\beta$, the linearization $\tau_{M_\phi,\lin}$ is described by the $n\times n$ matrix
\[
B:=\left(\begin{array}{cccc}0 &  & 0 & (t-b_0)/b_n \\1 &  & 0 & -b_1/b_n \\ & \ddots &  & \vdots \\0 &  & 1 & -b_{n-1}/b_n\end{array}\right) 
\]
where $\phi_t=\sum_{i=0}^n b_i \tau^i$.  

For $f\in \End_L(\phi)$, the map $M_\phi\to M_\phi,\, m\mapsto mf$ is a homomorphism of $L[t]$-modules which commutes with $\tau_{M_\phi}$.   This gives a homomorphism $\End_L(\phi)^{\text{opp}} \to \End(M_\phi)$ of $k[t]$-algebras; it is in fact an isomorphism \cite{MR850546}*{Theorem~1}.   Note that for a ring $R$, we will denote by $R^{\text{opp}}$ the ring $R$ with the same addition and multiplication $\alpha\cdot \beta= \beta\alpha$    

  Let $M_\phi(t)$ be the $\tau$-module obtained by base extending $M_\phi$ to $L(t)$.   Since $M_\phi(t)$ is an $L(t)$-vector space of dimension $n$ with $\det(B)\in L(t)^\times$, we find that $M_\phi(t)$ is an \'etale $\tau$-module.   We have an isomorphism $\End_L(\phi)^{\text{opp}} \otimes_{k[t]} k(t)=\End(M_\phi(t))$ of $k(t)$-algebras.  From Anderson \cite{MR850546}*{Prop.~4.1.1}, we know that $M_\phi(t)$ is pure of weight $1/n$ (use Lemma~\ref{L:lattice criterion} to relate his notion of purity and weight with ours).

Define $\bbar{M}_\phi:=\Lbar\twistedLaurent$.   For $c=\sum_i a_i t^{-i} \in \Lbar\laurent$ and $m\in \bbar{M}_\phi$,  we define
\[
c\cdot m = \sum_i a_i m \phi_t^{-i};
\]
this turns $\bbar{M}_\phi$ into a free $\Lbar\laurent$-module with basis $\{1,\tau,\ldots,\tau^{n-1}\}$.   The pair $(\bbar{M}_\phi,\tau_{\bbar{M}_\phi})$, where $\tau_{\bbar{M}_\phi}\colon \bbar{M}_\phi\to \bbar{M}_\phi$ is the map $m\mapsto \tau m$,  is a $\tau$-module.   One readily verifies that $\bbar{M}_\phi$ agrees with the base extension of $M_\phi$ to $\Lbar\laurent$.     

   Take any $f\in D_\phi$.   Since $f$ commutes with $\phi_t$, we find that the map $\bbar{M}_\phi\to \bbar{M}_\phi,\, m\mapsto mf$ is a homomorphism of $\Lbar\laurent$-modules which commutes with $\tau_{\bbar{M}_\phi}$.  This gives a homomorphism 
\begin{equation} \label{E:inclusion of division algebras}
D_\phi^{\text{opp}} \hookrightarrow \End(\bbar{M}_\phi)
\end{equation}
of $F_\infty$-algebras.    By Lemma~\ref{L:u facts}(\ref{I:u fact ii}) and Proposition~\ref{P:Dieudonne classification}, $D_\phi^{\text{opp}}$ and $\End(\bbar{M}_\phi)$ are both $F_\infty$-division algebras with invariant $1/n$, so (\ref{E:inclusion of division algebras}) is an isomorphism.  Moreover, the isomorphism (\ref{E:inclusion of division algebras}) is compatible with the respective $\Gal(\Lbar/L)$-actions.   Restricting (\ref{E:inclusion of division algebras}) to $\End_L(\phi)\otimes_{k[t]} k\laurent$ gives the isomorphism 
\[
\End_L(\phi)^{\text{opp}}\otimes_{k[t]} k\laurent \xrightarrow{\sim} \End(M_\phi) \otimes_{k[t]} k \laurent=\End(M_\phi(t)) \otimes_{k(t)} k \laurent.
\]
By Lemma~ \ref{L:reduction of Tate to invariants}, it suffices to prove that $D_\phi^{\Gal(\Lbar/L)}=\End_L(\phi)\otimes_{k[t]} k\laurent$, which we find is equivalent to showing that the natural map
\[
\End(M_\phi(t))\otimes_{k(t)} k\laurent \to \End(\bbar{M}_\phi)^{\Gal(\Lbar/L)} = \End(\Lbar\laurent\otimes_{L(t)} M_\phi(t))^{\Gal(\Lbar/L)}
\] 
is an isomorphism.   Since $M_\phi(t)$ is an \'etale $\tau$-module that is pure of weight $1/n$ and $L$ is finitely generated, this follows from Corollary~\ref{C:Tamagawa-Tate}.\\

\noindent{\textbf{Case 2:}} \emph{General case}.

Choose a non-constant element $t\in A$.   Composing the inclusion $k[t]\subseteq A$ with $\phi$ gives a ring homomorphism
\[
\phi' \colon k[t] \to L[\tau],\,\, a\mapsto \phi'_a.
\]
By (\ref{E:degrees}), we have $\ord_{\tau^{-1}}(\phi'_t) <0$ and hence $\phi'$ is a Drinfeld module (though possibly of a different rank than $\phi$).   Since $\phi$ has generic characteristic, so does $\phi'$.   Let $\infty$ also denote the place of $k(t)$ with uniformizer $t^{-1}$.  

Since $\phi(A)\supseteq \phi'(k[t])$, we have inclusions
\[
\End_L(\phi) \subseteq \End_L(\phi')\quad\text{ and }\quad D_\phi\subseteq D_{\phi'}.
\]
Therefore,
\[
\End_L(\phi) \otimes_A F_\infty \subseteq D_\phi^{\Gal(\Lbar/L)} \subseteq D_{\phi'}^{\Gal(\Lbar/L)} = \End_L(\phi') \otimes_A F_\infty
\]
where the equality follows from Case 1.     By Lemma~\ref{L:reduction of Tate to invariants}, it thus suffices to prove the inclusion $\End_L(\phi)\supseteq \End_L(\phi')$.    The ring $\End_L(\phi')$ certainly contains $\phi(A)$.    Since $\phi'$ has generic characteristic, the ring $\End_L(\phi')$ is commutative \cite{MR0384707}*{\S2}.  So $\End_L(\phi')$ is a subring of $L[\tau]$ that commutes with $\phi(A)$; it is thus a subset of $\End_L(\phi)$.

 \section{Proof of Theorem~\ref{T:main theorem non-CM}}  \label{S:proof of non-CM}
Let $\phi\colon A \to L[\tau]$ be a Drinfeld module of generic characteristic and rank $n$.   Assume that $L$ is a finitely generated field and that $\End_{\Lbar}(\phi)=\phi(A)$.    To ease notation, we set $D:=D_\phi$ which is a central $F_\infty$-division algebra with invariant $-1/n$.    Several times in the proof, we will replace $L$ by a finite extension; this is fine since we are only interested in $\rho_\infty(W_L)$ up to commensurability.   The $n=1$ case has already been proved (Corollary~\ref{C:main theorem n=1}), so we may assume that $n\geq 2$. 

\subsection{Zariski denseness}  \label{SS:Zariski density}
Let $\GL_D$ be the algebraic group defined over $F_\infty$ such that $\GL_D(R)=(D\otimes_{F_\infty} R)^\times$ for a commutative $F_\infty$-algebra $R$.   In particular, we have $\rho_{\infty}(W_L)\subseteq D^\times = \GL_D(F_\infty)$.  The main task of this section is to prove the following.

\begin{proposition} \label{P:Zariski density}
With assumptions as above,  $\rho_\infty(W_L)$ is Zariski dense in $\GL_D$.
\end{proposition}

Let $\GG$ be the algebraic subgroup of $\GL_D$ obtained by taking the Zariski closure of $\rho_\infty(W_L)$ in $\GL_D$; it is defined over $F_\infty$.   After replacing $L$ by a finite extension, we may assume that $\GG$ is connected.    Choose an algebraically closed extension $K$ of $F_\infty$.   For an algebraic group $G$ over $F_\infty$, we will denote by $G_K$ the algebraic group over $K$ obtained by base extension.   We need to prove that $\GG=\GL_D$, or equivalently that $\GG_K=\GL_{D,K}$.

An isomorphism $D\otimes_{F_\infty} K\cong M_n(K)$ of $K$-algebras induces an isomorphism $\GL_{D, K} \cong \GL_{n,K}$ of algebraic groups over $K$ (both are unique up to an inner automorphism).     We fix such an isomorphism, which we use as an identification $\GL_{D, K} = \GL_{n,K}$ and this gives us an action of $D\otimes_{F_{\infty}} K$ on $K^n$.

We will use the following criterion of Pink to show that $\GG_K$ and $\GL_{D,K}=\GL_{n,K}$ are equal. 

\begin{lemma}[\cite{MR1474696}*{Proposition~A.3}] \label{L:Pink criterion}
Let $K$ be an algebraically closed field and let $G\subseteq \GL_{n,K}$ be a reductive connected linear algebraic group acting irreducibly on $K^n$.  Suppose that $G$ has a cocharacter which has weight $1$ with multiplicity $1$ and weight $0$ with multiplicity $n-1$ on $K^n$.  Then $G=\GL_{n,K}.$
\end{lemma}

\begin{lemma} \label{L:irreducibility}
With our fixed isomorphism, the algebraic group $\GG_{K}$ acts irreducibly on $K^n$.
\end{lemma}
\begin{proof}
Let $B$ be the $F_\infty$-vector subspace of $D$ generated by $\rho_\infty(W_L)$.  Using that $\rho_\infty(W_{L})$ is a group and that every element of $D$ is algebraic over $F_\infty$, we find that $B$ is a division algebra whose center contains $F_\infty$.  By our analogue of the Tate conjecture (Theorem~\ref{T:Tate conjecture}) and our assumption $\End_{\Lbar}(\phi)=\phi(A)$,  we have
\[
\Cent_D(B) = \Cent_D(\rho_\infty(W_L)) = F_\infty.
\]
By the Double Centralizer Theorem, we have $B = \Cent_D(\Cent_D(B))$ and hence $B = \Cent_D(F_\infty) = D$.   

Let $H$ be a non-zero $K$-subspace of $K^n$ that is stable under the action of $\GG_{K}$.
Since $\rho_\infty(W_L)\subseteq \GG(F_\infty)$ and $F_\infty\subseteq K$, we find that $H$ is stable under the action of $B\otimes_{F_\infty} K=D\otimes_{F_\infty} K\cong M_n(K)$.   Therefore, $H=K^n$.
\end{proof}

By Lemma~\ref{L:irreducibility} and the following lemma, we deduce that $\GG_{K}$ is reductive.

\begin{lemma}[\cite{MR1474696}*{Fact~A.1}]
\label{L:Pink fact}
Let $K$ be an algebraically closed field, and let $G\subseteq \GL_{n,\, K}$ be a connected linear algebraic group.  If $G$ acts irreducibly on the vector space $K^n$, then $G$ is reductive. 
\end{lemma}

Let $X$ be a model of $L$ as described in \S\ref{SS:Sato-Tate law}.   For a fixed closed point $x$ of $X$, choose a matrix $h_x \in\GL_{n}(F)$ with characteristic polynomial $P_{\phi,x}(T)$.  Let $H_x\subseteq \GL_{n,F}$ denote the Zariski closure of the group generated by $h_x$, and let $T_x$ be the identity component of $H_x$.  Since $F$ has positive characteristic, some positive power of $h_x$ will be semisimple.  The algebraic group $T_x$ is thus an algebraic torus which is called the \defi{Frobenius torus} at $x$.   The following result of Pink describes what happens when $\phi$ has ordinary reduction at $x$.

Recall that by reducing the coefficients of $\phi$, we obtain a Drinfeld module $\phi_x \colon A \to \FF_x[\tau]$ of rank $n$.    Let $\lambda_x$ be the place of $F$ corresponding to the characteristic of $\phi_x$.  The Tate module $T_{\lambda_x}(\phi_x)$ is a free $\OO_{\lambda_x}$-module of rank  $n_x$, where $n_x$ is an integer \emph{strictly} less than $n$.  We say that $\phi$ has \defi{ordinary reduction} at $x$ if $n_x=n-1$.

\begin{lemma}[\cite{MR1474696}*{Lemma 2.5}] \label{L:Pink cocharacter}
If $\phi$ has ordinary reduction at $x\in X$, then $T_x\subseteq \GL_{n,F}$ possesses a cocharacter over $\Fbar$ which in the given representation has weight $1$ with multiplicity $1$, and weight $0$ with multiplicity $n-1$. 
\end{lemma}

\begin{lemma}[\cite{MR1474696}*{Corollary~2.3}] \label{L:ordinary reduction exists}
\label{L:ordinary}
The set of closed points of $X$ for which $\phi$ has ordinary reduction has positive Dirichlet density.
\end{lemma}

We can now finish the proof of Proposition~\ref{P:Zariski density}.   We have shown that $\GG_K$ is a reductive, connected, linear algebraic group acting irreducibly on $K^n$.  By Lemma~\ref{L:Pink criterion} it suffices to show that $\GG_K$ has a cocharacter which has weight $1$ with multiplicity $1$ and weight $0$ with multiplicity $n-1$ on $K^n$. 

By Lemma~\ref{L:ordinary reduction exists}, there exists a closed point $x$ of $X$ for which $\phi$ has ordinary reduction.  Some common power of $h_x$ and $\rho_\infty(\Frob_x)$ are conjugate in $\GL_n(K)$ because they will be semisimple with the same characteristic polynomial.   So with our fixed isomorphism $\GL_{D, K} = \GL_{n,K}$, we find that  $T_{x,K}$ is conjugate to an algebraic subgroup of $\GG_{K}$.   The desired cocharacter of $\GG_K$ is then obtained by appropriately conjugating the cocharacter coming from Lemma~\ref{L:Pink cocharacter}.

\subsection{Open commutator subgroup}  \label{SS:openness}
Let $\SL_D$ be the kernel of the homomorphism $\GL_D\to \GG_{m,F_\infty}$ arising from the reduced norm.   Let $\PGL_D$ and $\PSL_D$ be the algebraic groups obtained by quotienting $\GL_D$ and $\SL_D$, respectively, by their centers.   As linear algebraic groups, $\SL_D$ is simply connected and $\PGL_D$ is adjoint.   The natural map $\PSL_D\to \PGL_D$ is an isomorphism of algebraic groups and hence the homomorphism $\SL_D\to\PGL_D$ is a universal cover.

 The commutator morphism of $\GL_D$ factors through a unique morphism
\[
[\;,\;] \colon \PGL_D \times \PGL_D \to \SL_D.\\
\]
Let $\Gamma$ be the closure of the image of $\rho_\infty(W_L)$ in $\PGL_D(F_\infty)$.  Let $\Gamma' \subseteq \SL_D(F_\infty)$ be the closure of the subgroup generated by $[\Gamma,\Gamma]$.  (Both closures are with respect to the $\infty$-adic topology.)  

The group $\Gamma$ is compact since it is closed and $\PGL_D(F_\infty)$ is compact.   The group $\Gamma$ is Zariski dense in $\PGL_D$ by Proposition~\ref{P:Zariski density}.  If we were working over a local field of characteristic $0$, this would be enough to deduce that $\Gamma$ is an open subgroup of $\PGL_D(F_\infty)$.     However, in the positive characteristic setting the Lie theory is more complicated.  Fortunately, what we need has already been worked out by Pink.
  
Theorem 0.2(c) of \cite{MR1637068} says that there is a closed subfield $E$ of $F_\infty$, an absolutely simple adjoint linear group $H$ over $E$, and an isogeny $f\colon H\times_E F_\infty \to \PGL_D$ with nowhere vanishing derivative such that $\Gamma'$ is the image under $\widetilde{f}$ of an open subgroup of $\widetilde{H}(E)$ where
$\widetilde{f}\colon \widetilde{H}\times_E F_\infty \to \SL_D$ is the associated isogeny of universal covers. 

The following lemma will be needed to show that $E=F_\infty$.   Let 
\[
\Ad_{\PGL_D} \colon \PGL_D \to \GL_{m,F_\infty}
\]
be the adjoint representation of $\PGL_D$ where $m$ is the dimension of $\PGL_D$.

\begin{lemma} \label{L:adjoint field}
Let $\OO\subseteq F_\infty$ be the closure of the subring generated by $1$ and $\tr(\Ad_{\PGL_D}(\Gamma))$.  Then the quotient field of $\OO$ is $F_\infty$.
\end{lemma}
\begin{proof}
We will consider $\Ad_{\PGL_D}$ at Frobenius elements, and thus reduce to a result of Pink.  Take any element $\alpha \in D^\times$.   Let $\alpha_1,\dots, \alpha_n \in \Fbar_\infty$ be the roots of the (reduced) characteristic polynomial $\det(T I - \alpha)$.  We have
\[
\tr( \Ad_{\PGL_D}(\alpha))=  \Big( \sum_{i =1}^n {\alpha_i}\Big)\Big( \sum_{j =1}^n {\alpha_j}^{-1}\Big) - 1  = \tr(\alpha) \cdot \tr(\alpha^{-1}) - 1
\]
(one need only check the analogous result for $\PGL_{n,\Fbar_\infty}$ since it is isomorphic to $\PGL_{D,\Fbar_\infty}$).  For each closed point $x$ of $X$, define $a_x := \tr( \Ad_{\PGL_D}(\rho_\infty(\Frob_x)))$.   We have $a_x = \tr(\rho_\lambda(\Frob_x))\cdot \tr(\rho_\lambda(\Frob_x)^{-1}) -1$ for any place $\lambda \neq \lambda_x$ of $F$, and hence $a_x$ belongs to $F$.   By \cite{MR1474696}*{Proposition~2.4}, the field $F$ is generated by the set $\{a_x\}_x$ where $x$ varies over the closed points of $X$ (this requires our assumptions that $\End_{\Lbar}(\phi) =\phi(A)$ and $n\geq 2$).  Therefore, the quotient field of $\OO$ is $F_\infty$.
\end{proof}

Lemma~\ref{L:adjoint field} along with \cite{MR1637068}*{Proposition~0.6(c)} shows that $E=F_\infty$ and that $f\colon H \to \PGL_D$ is an isomorphism.   Therefore, $\Gamma'$ is an open subgroup of $\SL_D(F_\infty)$.

\begin{proposition} \label{P:openness}
The group $\rho_\infty(W_L)$ contains an open subgroup of $\SL_D(F_\infty)$.
\end{proposition}
\begin{proof}
The group $\rho_\infty(W_{L\kbar})$ is a normal subgroup of $\rho_\infty(W_{L})$ with abelian quotient; it is also compact since $W_{L\kbar}=\Gal(L^\sep/L\kbar)$ is compact and $\rho_\infty$ is continuous.   Therefore, $\Gamma'$ is a subgroup of $\rho_\infty(W_{L\kbar})$.      The proposition follows since we just showed that $\Gamma'$ is open in $\SL_D(F_\infty)$.
\end{proof}

\subsection{End of the proof}
We have $\rho_\infty(W_L)\subseteq D^\times$.  In Proposition~\ref{P:openness}, we showed that $\rho_\infty(W_L)$ contains an open subgroup of $\SL_D(F_\infty)=\{\alpha\in D^\times: \det(\alpha)=1\}$.  To complete the proof of Theorem~\ref{T:main theorem non-CM}, it suffices to show that $\det(\rho_\infty(W_L))$ is an open subgroup with finite index in $F_\infty^\times$.

\begin{lemma}
The image of the the reduced norm map $\det \circ \rho_{\infty} \colon W_L \to F_\infty^\times$ is an open subgroup of finite index in $F_\infty^\times$
\end{lemma}
\begin{proof}
One can construct a ``determinant'' Drinfeld module of $\phi$; it is a rank 1 Drinfeld module $\psi\colon A \to L[\tau]$ and has the property that $\bigwedge^n_{\OO_\lambda} T_\lambda(\phi)$ and $T_\lambda(\psi)$ are isomorphic $\OO_\lambda[\Gal_L]$-modules for every place $\lambda\neq \infty$ of $F$.     This can accomplished by following G.~Anderson and working in the larger category of \defi{$A$-motives} where one can take tensor products.   A proof of the existence of such a $\psi$ can be found in \cite{MR2097499}*{Theorem 3.3} and the isomorphism of Tate modules is then straightforward.  

Let $X$ be a model of $L$ as described in \S\ref{SS:Sato-Tate law}.     For each closed point $x$ of $X$ and place $\lambda\neq \lambda_x,\infty$ of $F$, we thus have 
\[
\det(\rho_{\phi,\infty}(\Frob_x))=\det(\rho_{\phi,\lambda}(\Frob_x))= \rho_{\psi,\lambda}(\Frob_x)=\rho_{\psi,\infty}(\Frob_x).
\]
By the Chebotarev density theorem, that $\det\circ \rho_{\phi,\infty}(\Frob_x)$ equals $\rho_{\psi,\infty}(\Frob_x)$ for all closed points $x$ of $X$ implies that $\det\circ \rho_{\phi,\infty}$ equals $\rho_{\psi,\infty}$.   The lemma now follows from Corollary~\ref{C:main theorem n=1} since $\psi$ has rank 1.
\end{proof}

\section{Proof of Theorem~\ref{T:main theorem}} \label{S:main theorem proof}
By \cite{MR0384707}*{p.569~Corollary} and our generic characteristic assumption, the ring $A':=\End_{\Lbar}(\phi)$ is a projective $A$-module and $F_\infty':=A' \otimes_A F_\infty$ is a field satisfying $[F_\infty':F_\infty]\leq n$.   Let $F'$ be the quotient field of $A'$.   There is a unique place of $F'$ lying over the place $\infty$ of $F$, which we shall also denote it by $\infty$, and $F_\infty'$ is indeed the completion of $F'$ at $\infty$.  

After replacing $L$ by a finite extension, we may assume that $A'$ equals $\End_{L}(\phi)$.  Identifying $A$ with its image in $L[\tau]$, the inclusion map
\[
\phi' \colon A' \to L[\tau].
\]
extends $\phi$.  The homomorphism $\phi'$ need not be a Drinfeld module, at least according to our definition, since $A'$ need not be a maximal order in $F'$.    Instead of extending our definition of Drinfeld module, we follow Pink and Hayes, and adjust $\phi'$ by an appropriate isogeny.   

Let $B$ be the normalization of $A'$ in $F'$; it is a maximal order of $F'$ consisting of functions that are regular away from $\infty$.  By \cite{MR535766}*{Prop.~3.2}, there is a Drinfeld module $\psi \colon B\to \Lbar[\tau]$ and a non-zero $f \in L[\tau]$ such that $f \phi'(x) = \psi(x) f$ for all $x\in A'$.  The Drinfeld module $\psi$ has rank $n'=n/[F':F]$ and $\End_{\Lbar}(\psi)=\psi(B)$.   After replacing $L$ by a finite extension, we may assume that $\psi(B)\subseteq L[\tau]$.

It is straightforward to show that the map $\Cent_{\Lbar\twistedLaurent}(\psi(A')) \to \Cent_{\Lbar\twistedLaurent}(\End_{\Lbar}(\phi))$ defined by $v\mapsto f^{-1}v f$ is a bijection, and hence we have an isomorphism
\[
D_\psi \xrightarrow{\sim} \Cent_{D_\phi}(\End_{\Lbar}(\phi))=:B_\phi,\quad v\mapsto f^{-1}v f.
\] 
The corresponding representations $\rho_\infty$ are compatible under this map.

\begin{lemma} 
For all $\sigma\in W_L$, we have $\rho_{\psi,\infty}(\sigma)= f^{-1} \rho_{\phi,\infty}(\sigma)f$.
\end{lemma}
\begin{proof}
Choose any $u\in \Lbar\twistedLaurent^\times$ such that $u^{-1}\psi(F'_{\infty})u \subseteq \kbar\twistedLaurent$.   So $u^{-1}f \phi(A) f^{-1}u \subseteq \kbar\twistedLaurent$ and hence $(f^{-1}u)^{-1} \phi(F_\infty) (f^{-1}u) \subseteq \kbar\twistedLaurent$.  Therefore, 
\[
\rho_{\psi,\infty}(\sigma)=\sigma(f^{-1}u)\tau^{\deg(\sigma)}(f^{-1}u)^{-1} = f^{-1}\sigma(u)\tau^{\deg(\sigma)}u^{-1}f=  f^{-1} \rho_{\phi,\infty}(\sigma)f.\qedhere
\]
\end{proof}

Therefore, $\rho_{\phi,\infty}(W_L)$ is an open subgroup of finite index in $\Cent_{D_\phi}(\End_{\Lbar}(\phi))^\times$ if and only if $\rho_{\psi,\infty}(W_L)$ is an open subgroup of finite index in $D_\psi^\times$.    However $\End_{\Lbar}(\psi)=\psi(B)$, so $\rho_{\psi,\infty}(W_L)$ is an open subgroup of finite index in $D_\psi^\times$ by Theorem~\ref{T:main theorem non-CM} which we proved in \S\ref{S:proof of non-CM}.

\section{Proof of Theorem~\ref{T:ST for surjective}} \label{S:equidistribution proof}
To ease notation, set $D=D_\phi$ and define the (surjective) valuation $v \colon D \to \ZZ\cup \{+\infty\},\, \alpha\mapsto \ord_{\tau^{-1}}(\alpha)/d_\infty$.    Let $\OO_D$ be the valuation ring of $D$ with respect to $v$ and let $\mathfrak{P}$ denote it maximal ideal.   We have fixed a uniformizer $\pi$ of $F_\infty$ that we can view as element of $D$ by identifying it with $\phi_\pi$.   Let $\mu_{D^\times}$ be a Haar measure for $D^\times$.  We fix an open subset $U$ of $\OO_\infty$, and let $\calC$ be the set of $\alpha\in D^\times$ for which $\tr(\alpha)\in U$.  We also fix an integer $0\leq i <n$, and let $\calV_i$ be the set of $\alpha\in D^\times$ that satisfy $v(\alpha)=-i$.  \\

  Take any positive integer $d\equiv i \pmod{n}$ that is divisble by $[\FF_L:\FF_\infty]$.  Let $x$ be a  closed point of $X$ of degree $d$.   We have $v(\rho_\infty(\Frob_x))= -\deg(\Frob_x)/d_\infty =- [\FF_x:\FF_\infty]=-d$.  Therefore,
\[
v(\rho_\infty(\Frob_x)\pi^{\lfloor d/n \rfloor}) = -d + {\lfloor d/n \rfloor} \ord_{\tau^{-1}}(\phi_\pi)/d_\infty  = - d + {\lfloor d/n \rfloor}n=-i.
\]
So $\rho_\infty(\Frob_x)\pi^{\lfloor d/n \rfloor}$ belongs to $\calV_i$; it belongs to $\calC$ if and only if $a_x(\phi) \pi^{\lfloor d/n \rfloor}$ is in $U$.

Let
\[
\bbar{\rho} \colon \Gal_L \xrightarrow{} D^\times/\ang{\pi}  
\]
be the continuous homomorphism obtained by composing $\rho_\infty \colon W_L \to D^\times$ with the quotient map to $D^\times/\ang{\pi}$, and then using the compactness of $D^\times/\ang{\pi}$ to extend by continuity.     

We can identify $\calV_i$, and hence also identify $\calV_i\cap \calC$, with its image in $D^\times/\ang{\pi}$.   This shows that
\[
\{ x \in |X|_d : a_x(\phi)\pi^{\lfloor d/n \rfloor} \in U \} =\{ x \in |X|_d : \bbar\rho(\Frob_x) \subseteq \calV_i\cap \calC \},
\]
which we can now estimate with the Chebotarev density theorem.  By assumption, we have $\rho_\infty(W_L)=D^\times$ and hence $\rho_\infty(W_{L\kbar})=\OO_D^\times$ by Lemma~\ref{L:u facts}(\ref{I:u fact iv}).   Therefore, $\bbar{\rho}(\Gal_L)=D^\times/\ang{\pi}$ and the cosets of $\bbar{\rho}(\Gal_{L\kbar})$ in $D^\times/\ang{\pi}$ are the sets $\calV_0,\calV_1,\ldots, \calV_{n-1}$.   By the global function field version of the Chebotarev density theorem, we have
\[
\lim_{\substack{d\equiv i \,(\!\bmod{n}),\, d\equiv 0 \,(\!\bmod{[\FF_L:\FF_\infty]}) \\ d \to + \infty}}\frac{|\{ x \in |X|_d : \bbar\rho(\Frob_x) \subseteq \calV_i\cap \calC \}|}{\#|X|_d} = \frac{\mu_{D^\times}(\calV_i\cap \calC)}{\mu_{D^\times}(\calV_i)}.
\]
It remains to compute the value ${\mu_{D^\times}(\calV_i\cap \calC)}/{\mu_{D^\times}(\calV_i)}$.

We first need to recall some facts about the division algebra $D$, cf.~\cite{MR0262250}*{\S2} for some background and references.  The algebra $D$ contains an unramified extension $W$ of $F_\infty$ of degree $n$ and an element $\beta$ such that
\[
D= W\oplus W\beta \oplus\cdots \oplus W \beta^{n-1}
\]
where $\beta^n$ is a uniformizer of $F_\infty$ and the map $a\mapsto \beta a \beta^{-1}$ generates $\Gal(W/F_\infty)$.     Define the map
\[
f\colon W^n \xrightarrow{\sim} D,\quad (a_0,\ldots, a_{n-1}) \mapsto {\sum}_{i=0}^{n-1} a_i \beta^i;
\]
it is an isomorphism of (left) vector spaces over $W$.  Let $\OO_W$ be the ring of integers of $W$ and denote its maximal ideal by $\mathfrak{p}$.   For any integers $m\in \ZZ$ and $0\leq j < n$, we have
\[
\mathfrak{P}^{mn+j}=f( (\mathfrak{p}^{m+1})^{j} \times (\mathfrak{p}^m)^{n-j} ).
\]
For $\alpha\in D$, the \emph{reduced trace} $\tr(\alpha)$ is the trace of the endomorphism of the $W$-vector space $D$ given by $v\mapsto v \alpha$.    One can check that $\tr(f(a_0,\ldots, a_{n-1}))=\Tr_{W/F_\infty}(a_0)$ for $(a_0,\ldots, a_{n-1})\in W^n$.    

First consider the case $i\geq 1$.  We have 
\[
\calV_i=\mathfrak{P}^{-i} - \mathfrak{P}^{-(i-1)}  =  f(\OO_W^{n-i} \times (\mathfrak{p}^{-1}-\OO_W)\times (\mathfrak{p}^{-1})^{i-1})
\]
and the measures arising from the restriction of the Haar measures of $D^\times$ and $W^n$, respectively, agree (up to a constant factor).   So
\begin{equation*} \label{E:trWF}
\frac{\mu_{D^\times}(\calV_i\cap \calC)}{\mu_{D^\times}(\calV_i)}= \mu_{\OO_W}(\{a_0\in \OO_W: \Tr_{W/F_\infty}(a_0)\in U\})
\end{equation*}
where $\mu_{\OO_W}$ is the Haar measure normalized so that $\mu_{\OO_W}(\OO_W)=1$.    Since $\Tr_{W/F_\infty}\colon \OO_W \to \OO_\infty$ is a surjective homomorphism of $\OO_\infty$-modules, we have $\mu_{\OO_W}(\{a_0\in \OO_W: \Tr_{W/F_\infty}(a_0)\in U\})=\mu(U)$.

Now consider the case $i=0$.  We have 
\[
\calV_0=\OO_D - \mathfrak{P}  =  f((\OO_W-\mathfrak{p})\times \OO_W^{n-1}).
\]
and the measures arising from the restriction of the Haar measures of $D^\times$ and $W^n$, respectively, agree (up to a constant factor).   So
\begin{equation*} \label{E:trWF}
\frac{\mu_{D^\times}(\calV_0\cap \calC)}{\mu_{D^\times}(\calV_0)}= \frac{\mu_{\OO_W}(\{a_0\in \OO_W-\mathfrak{p}: \Tr_{W/F_\infty}(a_0)\in U\})}{\mu_{\OO_W}(\OO_W-\mathfrak{p})}.
\end{equation*}
Note that $\Tr_{W/F_\infty}\colon \OO_W \to \OO_\infty$ is a surjective homomorphism of $\OO_\infty$-modules satisfying $\Tr_{W/F_\infty}(\mathfrak{p})=\pi\OO_\infty$.    Fix a coset $\kappa$ of $\pi\OO_\infty$ in $\OO_\infty$.   Then $\Tr_{W/F_\infty}^{-1}(\kappa) \cap (\OO_W-\mathfrak{p})$ is the union of $q^{d_\infty(n-1)}$ cosets of $\mathfrak{p}$ in $\OO_W$ when $\kappa\neq \pi \OO_\infty$, and $q^{d_\infty(n-1)}-1$ cosets when $\kappa=\pi\OO_\infty$.    One can then check that ${\mu_{\OO_W}(\{a_0\in \OO_W-\mathfrak{p}: \Tr_{W/F_\infty}(a_0)\in U\})}/{\mu_{\OO_W}(\OO_W-\mathfrak{p})}=\nu(U)$ by taking into account this weighting of cosets.

The following lemma will be used in the next section.
\begin{lemma} \label{L:div alg extra}
For $j\geq 1$, we have $\mu_{D^\times}(\{\alpha\in \OO_D-\pi \OO_D : \tr(\alpha)\equiv 0 \pmod{\pi^j\OO_\infty} \})\ll 1/q^{d_\infty j}$.
\end{lemma}
\begin{proof}
We have $f(\OO_W^n - \mathfrak{p}^n)= \OO_D-\pi \OO_D$.   One can then show that
\begin{align*}
&\mu_{D^\times}(\{\alpha\in \OO_D-\pi \OO_D : \tr(\alpha)\equiv 0 \pmod{\pi^j\OO_\infty} \})\\
 \ll& \mu'( \{(a_0,\ldots,a_{n-1})\in \OO_W^n-\mathfrak{p}^n: \Tr_{W/F_\infty}(a_0)\equiv 0 \pmod{\pi^j \OO_\infty}\})\\
 \ll& \mu_{\OO_W}(\{a_0\in \OO_W : \Tr_{W/F_\infty}(a_0)\equiv 0 \pmod{\pi^j \OO_\infty} \})
\end{align*}
where $\mu'$ is a fixed Haar measure of $W^n$.  This last quantity is bounded by $|\OO_\infty/\pi^j\OO_\infty|^{-1}=q^{-d_\infty j}$.
\end{proof}

\section{Proof of Theorem~\ref{T:LT bound}} \label{S:LT proof}

To ease notation, set $D=D_\phi$ and define the (surjective) valuation $v \colon D \to \ZZ\cup \{+\infty\},\, \alpha\mapsto \ord_{\tau^{-1}}(\alpha)/d_\infty$.    Let $\OO_D$ be the valuation ring of $D$ with respect to $v$.  Fix a uniformizer $\pi$ of $F_\infty$ that we can view as element of $D$ by identifying it with $\phi_\pi$.  

 For each $\alpha\in D^\times$, we define $e(\alpha)$ to be the smallest integer such that $\alpha\pi^{e(\alpha)}$ belongs to $\OO_D$ (equivalently, $v(\alpha\pi^{e(\alpha)})\geq 0$).  Define the map 
\[
f\colon D^\times \to \OO_\infty,\quad \alpha\mapsto \tr(\alpha\pi^{e(\alpha)})
\]  
where $\tr$ is the reduced trace.
For each integer $j\geq 1$, let $f_j \colon D^\times \to \OO_\infty/\pi^j\OO_\infty$ be the function obtained by composing $f$ with the reduction modulo $\pi^j$ homomorphism.

\begin{lemma} \label{L:LT bound 1}
Let $x$ be a closed point of $X$ of degree $d$.   Then $f_j(\rho_\infty(\Frob_x))=0$ for all integers $1\leq j \leq \ord_\infty(a_x(\phi))+\lceil d/n \rceil$.  In particular,
\[
P_{\phi,a}(d) \leq |\{x\in |X|_d: f_j(\rho_\infty(\Frob_x))=0 \}|.
\]
\end{lemma}
\begin{proof}
Set $\alpha := \rho_\infty(\Frob_x)$.  We have $v(\alpha)=-\deg(x)/d_\infty=-d$.  Since $v(\pi)=\ord_{\tau^{-1}}(\phi_\pi)/d_\infty= n$, we have $e(\alpha)=\lceil d/ n \rceil$.  Hence $f(\alpha) = \tr(\alpha\pi^{e(\alpha)}) = \tr(\alpha) \pi^{e(\alpha)} = a_x(\phi) \pi^{e(\alpha)}$, which is divisible by $\pi^j$ for any integer $1\leq j \leq \ord_\infty(a_x(\phi))+e(\alpha)$.
\end{proof}

For each integer $j\geq 1$, define the group
\[
G_j:=D^\times/(F_\infty^\times (1+\pi^j\OO_D)).
\]
If $\alpha,\beta \in D^\times$ are in the same coset of $G_j$, then $f_j(\alpha)=0$ if and only if $f_j(\beta)=0$ [observe that $f(\alpha\pi^i)=f(\alpha)$ for $i\in \ZZ$,  $f(u\alpha)= u f(\alpha)$ for $u \in \OO_\infty^\times$, and $f_j(\alpha (1+\pi^{j}\gamma))=f_j(\alpha)$ for $\gamma\in \OO_D$].  So by abuse of notation, it makes sense to ask whether $f_j(\alpha)=0$ for a coset $\alpha\in G_j$.  The subset $C_{j}:=\{ \alpha \in G_j : f_j(\alpha)=0\}$ of $G_j$ is stable under conjugation.    The group $G_j$ and the set $C_j$ do not depend on the initial choice of uniformizer $\pi$.  

Let $\bbar{\rho}_j \colon \Gal_L \to G_j$ be the Galois representation obtained by composing $\rho_\infty$ with the quotient map to $G_j$ and then extending to a representation of $\Gal_L$ by using that $\rho_\infty$ is continuous and $G_j$ is finite.   Lemma~\ref{L:LT bound 1} gives the bound  
\begin{equation} \label{E:LT start}
P_{\phi,a}(d) \leq |\{x\in |X|_d: \bbar\rho_j(\Frob_x)\subseteq C_j \}|
\end{equation}
whenever $1\leq j \leq \ord_\infty(a)+\lceil d/n \rceil$.   

We shall bound $P_{\phi,a}(d)$ by bounding the right-hand side of (\ref{E:LT start}) with an effective version of the Chebotarev density theorem and then choosing $j$ to optimize the resulting bound.   Let $\widetilde{G}_j$ be the image of $\bbar\rho_j\colon \Gal_L \to G_j$ and let $\widetilde{C}_j$ be the intersection of $\widetilde{G}_j$ with $C_j$.   The effective Chebotarev density theorem of Murty and Scherk \cite{MR1298275}*{Th\'eor\`eme~2} implies that 
\begin{align*}
|\{x\in |X|_d: \bbar\rho_j(\Frob_x)\subseteq \widetilde C_j \}| \ll m_j \frac{|\widetilde C_j|}{|\widetilde G_j|}\cdot  \#|X|_d 
+ |\widetilde C_j|^{1/2} (1+(\varrho_j+1)|\mathscr{D}|)\frac{q^{d_\infty d/2}}{d} 
\end{align*}
where the implicit constant depends only on $L$, and the quantities $m_j$, $|\mathscr{D}|$ and $\varrho_j$ will be described below.  (Their theorem is only given for a conjugacy class, not a subset stable under conjugation, but one can easily extend to this case by using the techniques of \cite{MR935007}.)

We first bound the cardinality of our subset $C_j$.

\begin{lemma} \label{L:CDT ingredients}
We have $|C_j| \ll q^{d_\infty (n^2-2)j}$ and $|C_j|/|G_j| \ll  1/q^{d_\infty j}$.
\end{lemma}
\begin{proof}
We first prove the bound of $|C_j|/|G_j|$.    For $\alpha\in D^\times$, we have $\alpha \pi^{e(\alpha)} \in \OO_D-\pi \OO_D$ and hence
\[
\frac{|C_j|}{|G_j|} \ll \mu_{D^\times}(\{\alpha \in \OO_D-\pi \OO_D : \tr(\alpha)\equiv 0 \pmod{\pi^j\OO_\infty}\})
\]
where $\mu_{D^\times}$ is a fixed Haar measure of $D^\times$.   From Lemma~\ref{L:div alg extra}, we deduce that ${|C_j|}/{|G_j|}\ll 1/q^{d_\infty j}$.

We have a short exact sequence of groups:
\[
1\to \OO_D^\times/ (\OO_\infty^\times(1+\pi^j\OO_D)) \to G_j \xrightarrow{v} \ZZ/n\ZZ \to 0.
\]
The group $\OO_D^\times/ (\OO_\infty^\times(1+\pi^j\OO_D))$ is isomorphic to $(\OO_D/\pi^j \OO_D)^\times/(\OO_\infty/\pi^j \OO_\infty)^\times$, and hence has cardinality 
\[
\frac{(q^{d_\infty n^2}-1) q^{d_\infty n^2 \cdot (j-1)} }{ (q^{d_\infty}-1)q^{d_\infty (j-1)}} = q^{d_\infty (n^2-1)j}\cdot \frac{1-1/q^{d_\infty n^2}}{1-1/q^{d_\infty}}.
\]
This proves that there are positive constants $c_1$ and $c_2$, not depending on $j$, such that $c_1q^{d_\infty(n^2-1)j} \leq |G_j| \leq c_2q^{d_\infty(n^2-1)j}$.  The required upper bound for $|C_j|$ follows from our bounds of $|C_j|/|G_j|$ and $|G_j|$.
\end{proof}

By Theorem~\ref{T:main theorem non-CM}, the index $[G_j:\widetilde{G}_j]$ can be bounded independent of $j$.    Lemma~\ref{L:CDT ingredients} and the inclusion $\widetilde{C}_j\subseteq C_j$ shows that ${|\widetilde C_j|}/{|\widetilde G_j|} \ll 1/q^{d_\infty j}$ and $|\widetilde{C}_j|^{1/2} \ll q^{d_\infty (n^2-2)j/2}$.

We define $L_j$ to be the fixed field in $L^\sep$ of the kernel of $\bbar\rho_j$.    Let $\mathcal C$ and $\mathcal C_j$ be smooth projective curves with function fields $L$ and $L_j$, respectively.  We can take $m_j:=[\FF_{L_j}:\FF_L]$ above, where $\FF_{L_j}$ and $\FF_{L}$ are the field of constants of $L_j$ and $L$, respectively.    Theorem~\ref{T:main theorem non-CM} implies that $\rho_\infty(\Gal(L^\sep/L\kbar))$ is an open subgroup of $\OO_D^\times$, and hence $m_j\leq [G_j : \bbar\rho_j(\Gal(L^\sep/L\kbar))]$ can be bounded independently of $j$.

We define $|\mathscr{D}|:= \sum_{x} \deg(x)$ where the sum is over the closed points of $\mathcal{C}$ for which the morphism $\mathcal{C}_j \to \mathcal{C}$,  corresponding to the field extension $L_j/L$, is ramified.    We may view $X$ as an open subvariety of $\mathcal{C}$.   Since the representation $\rho_\infty$ is unramified at all closed points of $X$ and $\mathcal{C}\setminus X$ is finite, we find that $|\mathscr{D}|$ can also be bounded independent of $j$.

Let $\mathcal{D}_{L_j/L}$ be the different of the extension $L_j/L$; it is an effective divisor of $\mathcal{C}_j$ of the form $\sum_{x} \sum_{y} d(y/x)\cdot y$, where the first sum is over the closed points $x$ of $\mathcal{C}$ and the second sum is over the closed points $y$ of $\mathcal{C}_j$ lying over $x$.   We define $\varrho_j$ to be the smallest non-negative integer for which the inequality $d(y/x) \leq e(y/x) (\varrho_j +1)$ always holds, where $e(y/x)$ is the usual ramification index.  We will prove the following bound for $\varrho_j$ in \S\ref{SS:proof of wild part}.
\begin{lemma} \label{L:wild part}  
With notation as above, we have $\varrho_j \ll j+1$ where the implicit constant does not depend on $j$.
\end{lemma}
Finally, we note that $\#|X|_d \ll q^{d_\infty d}/d$.  For any integer $1\leq j \leq \ord_\infty(a)+\lceil d/n \rceil$, combining all our bounds together we obtain
\begin{align*}
P_{\phi,a}(d)& \ll \,  \frac{1}{q^{d_\infty j}} \frac{q^{d_\infty d}}{d} + q^{d_\infty(n^2-2)j/2}\cdot j \cdot \frac{q^{d_\infty d/2}}{d} = \, \frac{q^{d_\infty(d-j)}}{d} + q^{d_\infty((n^2-2)j+ d)/2} \cdot \frac{j}{d}
\end{align*}
where the implicit constant depends only on $\phi$.   We choose $j:=\ord_\infty(a)+\lceil{d/ n^2}\rceil$; for $d$ sufficiently large, we do indeed have $1\leq j \leq \ord_\infty(a)+\lceil d/n \rceil$.  With this choice of $j$, we obtain the desired bound $P_{\phi,a}(d) \ll q^{d_\infty(1-1/n^2)d}$.  

\subsection{Proof of Lemma~\ref{L:wild part}}\label{SS:proof of wild part}
Fix a non-constant $y\in A$ and define $h=-nd_\infty\ord_\infty(y)\geq 1$.   Construct $\delta \in (L^\sep)^\times$ and $a_0=1, a_1,a_2,\ldots\in L^\sep$ as in the beginning of \S\ref{S:construction}.   The series $u=\delta(\sum_{i=0}^\infty a_i \tau^{-i})$ then satisfies $u^{-1}\phi(F_\infty) u \subseteq \kbar\twistedLaurent$.   We noted above that $[\FF_{L_j}:\FF_L]$ can be bounded independently of $j$.   So there is a finite subfield $\FF$ of $\kbar$ that contains all the fields $\FF_{L_j}$ and also the field with cardinality $q^h$.  Set $K_0=L\FF(\delta)$, and recursively define the subfields $K_{i+1}:=K_i(a_{i+1})$ of $L^\sep$ for $i\geq 0$.    For $\sigma\in \Gal(L^\sep/L\kbar)$, we have $\rho_\infty(\sigma) \in 1+\pi^j \OO_D$ if only if $v(\rho_\infty(\sigma)-1)=v(\sigma(u)u^{-1} -1) = v(\sigma(u)-u)$ is greater than or equal to $v(\phi_\pi^j)= j n$.  This implies that $L_j$ is a subfield of $K_{jn}$.\\

Consider a chain of global function fields $F_1\subseteq F_2 \subseteq F_3$ with valuations $v_1$, $v_2$, and $v_3$, respectively (so $v_3$ lies over $v_2$ and $v_2$ lies over $v_1$).    We then have $d(v_3/v_1)= e(v_3/v_2)d(v_2/v_1)+d(v_3/v_2)$, equivalently
\begin{equation} \label{E:diff ram}
\frac{d(v_3/v_1)}{e(v_3/v_1)} = \frac{d(v_2/v_1)}{e(v_2/v_1)}+\frac{d(v_3/v_2)}{e(v_3/v_1)},
\end{equation}
where $d(v_j/v_i)$ is the degree of the different $\mathcal{D}_{F_j/F_i}$ at $v_j$ and $e(v_j/v_i)$ is the usual ramification index.    
Fix an integer $j$, and take any place $v$ of $L$ and any place $w$ of $L_j$ lying over $v$.    Since $L \subseteq L_j \subseteq K_{jn}$, we can choose a place $w'$ of $K_{jn}$ lying over $w$.  Using (\ref{E:diff ram}), we have
\[
\frac{d(w/v)}{e(w/v)}  \leq \frac{d(w'/v)}{e(w'/v)}.
\]

It thus suffices to prove that $d(w/v)/e(w/v)\ll j+1$ holds for every place $v$ of $L$, $j\geq 0$, and place $w$ of $K_{j}$ lying over $v$.   Fix a place $v$ of $L$.

\begin{lemma} \label{L:ord lower}
There is a constant $B\geq 0$ such that $\ord_v(a_i)\geq -B$ holds for all $i\geq 0$ and all valuations $\ord_v\colon L^\sep \to \QQ\cup\{+\infty\}$ extending $\ord_v$.
\end{lemma}
\begin{proof}  
From (\ref{E:artin-schreier}), we find that 
\begin{equation} \label{E:artin-schreier 2}
\frac{1}{q^h}\ord_v(a_i^{q^h}-a_i) \geq -C + \min_{\substack{0\leq j\leq h-1\\ i+j-h\geq 0}} \frac{\ord_v(a_j)}{q^{h-j}}
\end{equation}
holds for some constant $C\geq 0$.   Define $B:= C/(1-1/q)$.

  We will proceed by induction on $i$.  The lemma is trivial for $i=0$, since $\ord_v(a_0)=0$.   Now take $i\geq 1$.   If $\ord_v(a_i)\geq 0$, then we definitely have $\ord_v(a_i)\geq -B$.  Suppose that $\ord_v(a_i)<0$.  Then the roots of (\ref{E:artin-schreier}) as a polynomial in $a_i$ are $a_i+b$ with $b$ in the subfield of $\kbar$ of cardinality $q^h$; we have $\ord_v(a_i+b)=\ord_v(a_i)$ for all such $b$ (since $\ord_v(a_i)<0$), so  $\ord_v(a_i) = \ord_v(a_i^{q^h}-a_i)/q^h$.  By (\ref{E:artin-schreier 2}) and our inductive hypothesis, we deduce that
\[
\ord_v(a_i) \geq -C - B/q = -B(1-1/q)-B/q=-B. \qedhere
\]
\end{proof}

\begin{lemma} \label{L:exp bound}
For a fixed integer $j\geq 0$, let $w_{j}$ and $w_{j+1}$ be places of $K_{j}$ and $K_{j+1}$, respectively, such that $w_j$ lies over $v$ and $w_{j+1}$ lies over $w_j$.  We then have
\[
d(w_{j+1}/w_j) \leq C e(w_{j+1}/v)
\]
where $C$ is a non-negative constant that does not depend on $j$.
\end{lemma}
\begin{proof}
Choose an extension $\ord_v\colon L^\sep \to \ZZ\cup\{+\infty\}$ that corresponds to $w_{j+1}$ when restricted $K_{j+1}$.  We defined $a_{j+1}$ to be a root of the polynomial $X^{q^h}-X + \beta_j$ with $\beta_j \in K_j$.  Let $k_h$ be the subfield (of $K_0$) of cardinality $q^h$.  We have $X^{q^h}-X + \beta_j= \prod_{b \in k_h} (X- a_{j+1}+b)$, so for each $\sigma \in \Gal_{K_{j}}$, there is a unique $\gamma(\sigma) \in k_h$ such that $\sigma(a_{j+1})= a_{j+1} + \gamma(\sigma)$.  Since $k_h \subseteq K_{j}$,  we find that $\gamma \colon \Gal_{K_{j}}\to k_h$ is a homomorphism whose image we will denote by $H$.    Define the additive polynomial $g(X):= \prod_{b\in H} (X+b) \in k_h [X]$.   The minimal polynomial of $a_{j+1}$ over $K_{j}$ is thus
\[
g(X-a_{j+1}) = g(X) - g(a_{j+1}) \in K_{j}[X],
\]
and the extension $K_{j+1}/K_{j}$ is Galois with Galois group $H$.   

The extension $K_{j+1}/K_j$  is a variant of the familiar Artin-Schreier extensions.   If $\ord_v(g(a_{j+1}))\geq 0$, then $K_{j+1}/K_j$ is unramified at $w_j$ \cite{St}*{Prop.~3.7.10(c)}, so $d(w_{j+1}/w_j)=0$ and the lemma is trivial.   So we may suppose that $m:=-\ord_v(g(a_{j+1}))>0$ and that $K_{j+1}/K_j$ is ramified at $w_j$.  We then find that $K_{j+1}/K_j$ is totally ramified at $w_j$ and that $d(w_{j+1}/w_{j}) \leq (|H|-1)(m+1) e(w_{j}/v)$ (see \cite{St}*{Prop.~3.7.10(d)}; the factor $e(w_{j}/v)$ arises by how we normalized our valuation).  Therefore, $d(w_{j+1}/w_j) \leq (m+1) e(w_{j+1}/v)$.    It thus suffices to prove that $\ord_v(g(a_{j+1}))$ can be bounded from below by some constant not depending on $j$; this follows immediately from Lemma~\ref{L:ord lower}.
\end{proof}

We finally prove that $d(w/v)/e(w/v)\ll j+1$ holds for every place $v$ of $L$, $j\geq 0$, and place $w$ of $K_{j}$ lying over $v$.   If the place $v$ corresponds to one of the closed points of $X$, then we know that $\rho_\infty$, and hence $K_j$, is unramified at $v$; so $d(w/v)/e(w/v)=0$.     We may now fix $v$ to be a one of the finite many places of $L$ for which $\rho_\infty$ is ramified. 

Fix a positive constant $C$ as in Lemma~\ref{L:exp bound}.   After possibly increasing $C$, we may assume that  $d(w_0/v)\leq C e(w_0/v)$ holds for every place $w_0$ of $K_0$ lying over $v$.     Take any places $w_j$ of $K_j$ for $j\geq 0$ such that $w_{j+1}$ lies over $w_j$ and $w_0$ lies over $v$.    By (\ref{E:diff ram}) and Lemma~\ref{L:exp bound}, we have
\[
\frac{d(w_{j+1}/v)}{e(w_{j+1}/v)} = \frac{d(w_j/v)}{e(w_j/v)} + \frac{d(w_{j+1}/w_j)}{e(w_{j+1}/v)} \leq \frac{d(w_j/v)}{e(w_j/v)}  + C.
\]
Since $d(w_0/v)/e(w_0/v)\leq C$ by our choice of $C$, it is now easy to show by induction on $j$ that ${d(w_j/v)}/{e(w_j/v)}\leq C(j+1)$ holds for all $j\geq 0$.



\end{document}